\documentclass[11pt]{amsart}

\usepackage{amsmath,amssymb,amsthm}
\usepackage{thmtools,enumerate}
\usepackage{mathtools}
\usepackage{todonotes}
\usepackage{enumitem}

\usepackage[german,english]{babel}
\usepackage[autostyle]{csquotes}

\usepackage[all]{xy}
\usepackage{pstricks}

\usepackage{tikz-cd}
\usetikzlibrary{babel}
\usepackage{hyperref}
\hypersetup{
	colorlinks=true,
	linkcolor=red,
	citecolor=blue}

\theoremstyle{plain}

\newtheorem{thm}{Theorem}[section]
\newtheorem{prop}[thm]{Proposition}
\newtheorem{lem}[thm]{Lemma}
\newtheorem{cor}[thm]{Corollary}

\newtheorem{conj}[thm]{Conjecture}

\theoremstyle{definition}

\newtheorem{dfn}[thm]{Definition}
\newtheorem{rem}[thm]{Remark}
\newtheorem{exa}[thm]{Example}
\newtheorem{notation}[thm]{Notation}

\newcommand{\Z}{\mathbb{Z}}
\newcommand{\N}{\mathbb{N}}
\newcommand{\Q}{\mathbb{Q}}
\newcommand{\R}{\mathbb{R}}

\newcommand{\OO}{\mathcal{O}}
\newcommand{\pr}{\mathbb{P}}

\newcommand{\Id}{\mathrm{Id}}

\newcommand{\GL}[0]{\operatorname{GL}}

\newcommand{\Aut}[0]{\operatorname{Aut}}

\newcommand{\cal}[1]{\mathcal{#1}}
\newcommand{\id}{\mathrm{id}}

\DeclareMathOperator{\Int}{Int}

\DeclareMathOperator{\Pic}{Pic}
\DeclareMathOperator{\Supp}{Supp}

\DeclareMathOperator{\NE}{NE}
\DeclareMathOperator{\Nef}{Nef}
\DeclareMathOperator{\NefE}{Nef^e}
\DeclareMathOperator{\NefP}{Nef^+}
\DeclareMathOperator{\Amp}{Amp}

\DeclareMathOperator{\Eff}{Eff}
\DeclareMathOperator{\Psef}{\overline{Eff}}
\DeclareMathOperator{\Bl}{Bl}
\DeclareMathOperator{\Cr}{Cr}
\DeclareMathOperator{\Bs}{Bs}
\DeclareMathOperator{\Mov}{Mov}
\DeclareMathOperator{\MovE}{\overline{Mov}^e}
\DeclareMathOperator{\MovP}{\overline{Mov}^+}
\DeclareMathOperator{\MovC}{\overline{Mov}}
\DeclareMathOperator{\conv}{conv}
\DeclareMathOperator{\PsAut}{PsAut}

\begin{document}
	
	\title{Cones of divisors on $\mathbb{P}^3$ blown up at eight very general points}
	
	\author[I.\ Stenger]{Isabel Stenger}
	\address{Institute of Algebraic Geometry, Leibniz University Hannover, Welfengarten 1, 30167 Hannover, Germany.}
	\email{stenger@math.uni-hannover.de}
		
	\author[Z.\ Xie]{Zhixin Xie}
    \address{Institut \'Elie Cartan de Lorraine, Universit\'e de Lorraine, 54506 Nancy, France}
    \email{zhixin.xie@univ-lorraine.fr}


	\begin{abstract} Let $X$ be the blowup of $\pr^3$ at eight very general points. We give a complete description of its nef and effective cones.
	Moreover, we show that there exists a rational polyhedral fundamental domain for the action of a certain Weyl group on the effective movable cone of $X$.
		\end{abstract}

	\maketitle

	\begingroup
		\hypersetup{linkcolor=black}
		\setcounter{tocdepth}{1}
		\tableofcontents
	\endgroup

\section{Introduction}

Mori’s seminal work in the 1980s revealed a deep connection between the minimal models of a projective variety $X$ and the structure of the cone of nef effective divisors and the cone of movable effective divisors on $X$.
When $X$ is a Fano variety, or more generally a Mori dream space, Hu and Keel proved in \cite{HK00} that the nef cone is rational polyhedral, and the movable cone is the union of nef cones of finitely many small modifications of $X$. Moreover, both of these cones are contained in the cone of effective divisors, which is rational polyhedral.

When $X$ is a Calabi-Yau variety, these cones have a much more complicated structure. Nevertheless, the Morrison-Kawamata cone conjecture (\cite{Mor93}, \cite{Kaw97}) predicts that the groups of automorphisms and pseudo-automorphisms act with rational polyhedral fundamental domain on the effective nef and the effective movable cones of the variety respectively. Despite of being verified in several instances, the cone conjecture for Calabi-Yau varieties remains widely open; we refer the reader to \cite{LOP18} for a detailed survey on this topic. We stress that the effective cone of the variety plays an important role in the recent development on the Morrison-Kawamata cone conjecture; see \cite[Conjecture 1.3 and Theorem 1.4]{GLSW26}, see also  \cite[Conjecture 1.2]{LZ25} and \cite{Li23} for the relative setting.

In general, it is difficult to describe the nef, the movable and the effective cones of a variety, even under weaker assumptions on the anticanonical class $-K_X$.
In this article, we consider the case where $-K_X$
 is nef. This class of varieties interpolates between Fano varieties and varieties with numerically trivial canonical class, and we may therefore hope that the effective nef and effective movable cones admit a description similar to that predicted by the cone conjecture.
We note that when $-K_X$ is semiample, there are several examples where the 
 effective nef and the effective movable cones are fully described, see Subsection \ref{subsec:conjpairs}.
\subsection{Blowups of $\pr^n$ and the Weyl group}\label{subsec:blowupweyl} 

Let $X_{n,k}$ denote the blowup of $\pr^n$ in $k > n+1$ points in general position, Dolgachev \cite{Dol83} realised a root system in the group $N^1(X_{n,k})$ and interpreted the action of the Weyl group ${\cal W}_{n,k}$ corresponding to this root system by means of Cremona transformations.
Mukai \cite{Mukai01}, Castravet and Tevelev \cite{CT06} showed that the monoid of effective divisor classes in $\Pic(X_{n,k})$ is finitely generated if and only if
\begin{equation*}
\frac{1}{n+1} + \frac{1}{k-n-1} > \frac{1}{2}.
\end{equation*} 
We outline their proof as follows.
When the above inequality holds, the corresponding Weyl group ${\cal W}_{n,k}$ is finite and the effective cone $\Eff(X_{n,k})$ is spanned as a convex cone by the finitely many classes in ${\cal W}_{n,k}\cdot E_k$ (see \cite[Theorem 2.7]{CT06}), where $E_k$ is an exceptional divisor of the blowup $X_{n,k}$ of $\pr^n$. Thus $\Eff(X_{n,k})$ is a rational polyhedral cone. 

In the case of an infinite Weyl group, the effective cone $\Eff(X_{n,k})$ is only poorly understood:
Mukai showed in \cite{Mukai01} that the classes in the orbit ${\cal W}_{n,k} \cdot E_k$ are contained in any generating system of the effective cone.  If $n=3$, then $k=8$ is the smallest number which yields an infinite Weyl group ${\cal W}_{3,8}$, and hence an effective cone which is not rational polyhedral.
We show that the effective cone is closed and explicitly determine its infinitely many extremal rays:
\begin{thm}\label{main:eff} 
Let $X$ be the blowup of $\pr^3$ at eight very general points and let ${\cal W} = {\cal W}_{3,8}$. Then $\Eff(X)=\overline{\Eff}(X)$, and
$\Eff(X)$ is the convex cone generated by the classes in ${\cal W} \cdot E_8$ and the class $-\frac{1}{2}K_X$, where $E_8$ is an exceptional divisor.
\end{thm}
Our result above shows that, in contrast to the case where $X_{n,k}$ is a Mori dream space, the (infinitely many) classes in the 
${\cal W}_{3,8}$-orbit of an exceptional divisor do not suffice to generate the effective cone.

\subsection{Cone conjecture on Calabi-Yau pairs}\label{subsec:conjpairs}
Totaro introduced a generalisation of the cone conjecture to klt Calabi-Yau pairs (see \cite{Tot08}) and verified the conjecture in dimension two (see \cite{Tot10}).  Moreover, using the example of $\pr^2$ blown up at nine very general points (see \cite{Nag60}), Totaro showed that the conjecture is in general not true for log canonical Calabi-Yau pairs. However, using a different group action on the effective nef cone, Li verified in \cite{Li25} a variant of the cone conjecture on the effective nef cone for a class of log canonical Calabi-Yau surface pairs.

When $X$ has a semiample anticanonical class, one can relate the description of the effective nef and the effective movable cones of $X$ to the cone conjecture for klt Calabi-Yau pairs: since $-K_X$ is semiample, choosing $m$ large enough so that $-mK_X$ is base-point-free, a general divisor $D\in|{-}mK_X|$ is smooth. Thus putting $\Delta=\frac{1}{m}D$, the pair $(X,\Delta)$ is a klt Calabi-Yau pair. For many examples of threefolds with semiample anticanonical class, the
cone conjecture for klt Calabi-Yau pairs has been verified; see \cite{Ps12,PS15} and \cite{CPS14,CPS19}. 

  Now  let $X $ be the blowup of $\pr^3$ at eight very general points. Then $-K_X$ is nef and \textbf{not} semiample, as shown in \cite{LesOtt16}. 
  Choosing a general divisor $\Delta\in |{-}K_X|$, the pair $(X,\Delta)$ is a log canonical Calabi-Yau pair. 
  To verify the cone conjecture for the effective movable cone $\MovE(X)$,
  one major difficulty is to determine the pseudo-automorphisms of $X$. Our strategy to overcome this problem is to describe this cone with respect to a different subgroup of $\GL\big( N^1(X)\big)$ preserving $\MovE(X)$, which is related to the geometry of $X$.
 \begin{thm}\label{main:mov}
Let $X$ and ${\cal W}$ be as in Theorem \ref{main:eff}. Then
\begin{enumerate}[leftmargin=*]
		\item We have $\MovE(X)=\MovP(X) =\MovC(X)$. 
		\item The cone $\MovE(X) $ is not rational polyhedral, but there exists a rational polyhedral cone which is a fundamental domain for the action of ${\cal W}$ on $\MovE(X) $.
	\end{enumerate}
\end{thm}
As a crucial step to prove Theorem \ref{main:mov}, we construct an explicit rational polyhedral cone $\Pi$ (see Definition \ref{def:conePi}) whose orbit covers the interior of the movable cone $\MovC(X)$, that is
\[
  \Mov(X) \subset \cal W \cdot \Pi. 
\]
Moreover, the proof relies on the following decomposition result of the effective movable cone:
\begin{prop}\label{prop:main}
Let $X$ be a smooth projective threefold such that $-K_X$ is nef. Then
\[
\MovP(X)\supset \MovE(X) = \bigcup_{\varphi\colon X\dashrightarrow X' \text{ a finite sequence of flops}} \varphi^* \big(\NefE(X') \big).
\]
\end{prop}
We point out that 
\cite[Proposition 1.1]{GLSW26} showed that the effective cone of any klt Calabi--Yau pair of dimension three admits a decomposition into Mori chambers. Although varieties with nef anticanonical class do not, in general, fall within this framework, it is natural to ask whether the decomposition of the effective movable cone $\MovE(X)$ obtained in Proposition \ref{prop:main} extends to a Mori chamber decomposition of the entire effective cone $\Eff(X)$.
On the other hand, for blowups of projective space $X_{n,k}$ as in Subsection \ref{subsec:blowupweyl}, \cite[Definition 7.6]{BDPS24} introduced a Weyl chamber decomposition of the pseudoeffective cone $\Psef(X_{n,k})$, which coincides with the Mori chamber decomposition whenever $X_{n,k}$ is a Mori dream space. Moreover, \cite[Conjecture 7.7]{BDPS24} predicts that, for the blowup $X_{3,8}$ of $\mathbb{P}^3$ at eight very general points, the Weyl chamber decomposition of $\Eff(X_{3,8})$ coincides with its Mori chamber decomposition, provided the latter exists.

For $X = X_{3,8}$, it is shown in \cite{CLO16} that the cone of curves of $X$ is rational polyhedral. As its dual cone, the nef cone of $X$ is also rational polyhedral. We construct in Lemma \ref{lem:nefcone} an explicit set of divisor classes generating the nef cone $\Nef(X)$, which are all effective. Therefore, the effective nef cone conjecture holds for $X$, see Proposition \ref{main:nef}.

Motivated by this paper,  Gachet recently verified in \cite{Gach25} that the pseudo-automorphism group of $X$ is trivial.
Consequently, Theorem \ref{main:mov} provides another example where a generalisation of the cone conjecture to log canonical Calabi-Yau pairs fails. In this sense, Theorem \ref{main:mov} can be seen as a variant of the cone conjecture for the effective movable cone with respect to a different group.

\subsection*{Acknowledgments} 
We would like to thank Vladimir Lazi\'c for his helpful comments and discussions leading to Theorem \ref{thm:g-mmp-movcone}. We are grateful to Artie Prendergast-Smith for his inspiring ideas for the proof of Proposition \ref{prop:effconeclosed}. We would like to thank 
 C\'ecile Gachet, Andreas H\"oring, John Lesieutre and Susanna Zimmermann for useful discussions. We thank the anonymous referee for valuable comments and suggestions.
The first author  is supported by the Deutsche Forschungsgemeinschaft (DFG, German Research Foundation) - Project-ID 286237555 - TRR 195.

\section{Notation and preliminaries}

In this paper we work over the field $\mathbb{C}$.

\subsection{Notation}
First we fix some notation.

Let $X$ be a normal projective variety.
Let $N^1(X)$ be the Néron-Severi group generated by the classes of Cartier divisors modulo numerical equivalence; let $N^1(X)_{\R}=N^1(X)\otimes \R$. Inside $N^1(X)_{\R}$, we have:
\begin{itemize}[leftmargin=*]
\item the effective cone $\Eff(X)$ and the pseudoeffective cone $\Psef(X)$;
\item the movable cone $\MovC(X)$, which is the closure of the cone generated by classes of effective Cartier divisors $D$ such that the base locus of $|D|$ has codimension at least $2$; and its interior $\Mov(X)$;
\item the effective movable cone $\MovE(X)\coloneqq \MovC(X)\cap \Eff(X)$;
\item the convex hull $\MovP(X)$ of all rational classes in $\MovC(X)$;
\item the nef cone $\Nef(X)$; and its interior $\Amp(X)$;
\item the effective nef cone $\NefE(X)\coloneqq \Nef(X)\cap\Eff(X)$;
\item the convex hull $\NefP(X)$ of all rational classes in $\Nef(X)$.
\end{itemize}

We call an $\R$-divisor $D$ \textit{movable} if its class is in $\MovC(X)$. 

We denote by $N_1(X)$ the group of $1$-cycles modulo numerical equivalence and $N_1(X)_{\R}=N_1(X)\otimes \R$.
Inside $N_1 ( X )_{\R}$, we consider the convex cone $\NE(X)$ generated by effective 1-cycles and its closure, the cone of curves, $\overline{\NE}(X)$. 

We denote by $\Aut(X)$ and $\PsAut(X)$ the group of automorphisms and pseudo-automorphisms of $X$, that means birational maps $X  \dashrightarrow X$ which are isomorphisms in codimension one. Any automorphism induces an isomorphism on $N^1(X)_{\R}$ which preserves the nef cone, whereas any pseudo-automorphism induces an isomorphism on $N^1(X)_{\R}$  preserving the effective and the movable cones. We denote by $\Aut(X)^*$ and $\PsAut(X)^*$ the image of $\Aut(X)$ and $\PsAut(X)$ in $\GL(N^1(X)_{\R})$, respectively.

Finally, we say that an $\R$-divisor on $X$ is an {\em NQC divisor} if it is a non-negative linear combination of nef $\Q$-Cartier divisors on $X$. The acronym NQC stands for {\em nef $\Q$-Cartier combinations}. In particular, if $X$ is $\Q$-factorial, then the set of all NQC divisors on $X$ is equal to $\Nef^+(X)$.
\subsection{The Morrison-Kawamata cone conjecture}
In this subsection we state a generalisation of the Morrison-Kawamata cone conjecture for klt Calabi-Yau pairs. 

A \textit{pair} consists of a normal variety $X$ and an effective $\R$-divisor $\Delta$ on $X$ such that the divisor $K_X+\Delta$ is $\R$-Cartier. Given a $\Q$-factorial projective klt (respectively log canonical) pair $(X,\Delta)$, we say that the pair is \textit{Calabi-Yau} if $K_X + \Delta$ is numerically trivial.

\begin{conj}{\cite{Tot08}} Let $(X,\Delta)$ be a klt Calabi-Yau pair. Then:
	\begin{enumerate}[leftmargin=*]
		\item There exists a rational polyhedral cone $\Pi$ which is a fundamental domain for the action of $\Aut(X)$ on $\NefE(X) = \Nef(X) \cap \Eff(X)$, in the sense that
		\[
		\NefE(X)= \bigcup_{g\in \Aut(X)} g^* \Pi
		\] 
		with $\Int \Pi  \cap \Int g^* \Pi = \emptyset$ unless $g^* = \id$. 
		\item There exists a rational polyhedral cone $\Pi'$ which is a fundamental domain for the action of $\PsAut(X)$ on $\MovE(X) = \MovC(X) \cap \Eff(X)$.
	\end{enumerate}
\end{conj}
\begin{rem} We remark that in \cite{Tot10} there exists a different formulation of the Morrison-Kawamata cone conjecture for klt Calabi-Yau pairs using the group $\Aut(X,\Delta)$ (resp.\ $\PsAut(X,\Delta)$) of automorphisms (resp.\ pseudo-automorphisms) of $X$ preserving the boundary divisor $\Delta$. As $\Aut(X,\Delta) \subset \Aut(X)$ and analogously for the pseudo-automorphisms, the version in \cite{Tot10} is stronger that the one stated above. 
\end{rem} 
A useful tool to show the existence of a fundamental domain for a group action on a convex cone is the following result of Looijenga (see also \cite[Proposition 3.3 and Lemma 3.5]{LZ25}):
\begin{thm}{\cite[Theorem 4.1 and Application 4.15]{Loo14}}\label{thm:looi}
	Let $V$ be a real vector space with rational structure $V(\Q)$, $C$ an open strictly convex cone in $V$, and let $G$ be a subgroup of $\GL(V)$ stabilising the cone $C$. Let $C_+$ denote the convex hull of $\overline{C} \cap V(\Q)$. Then the following are equivalent:
	\begin{enumerate}[leftmargin=*]
		\item There exists a rational polyhedral cone $\Pi$ in $C_{+}$ with 
		$C_{+} = G \cdot \Pi$. 
		\item There exists a rational polyhedral cone $\Pi$ in $C_{+}$ with 
		$C \subset G \cdot \Pi$. 
	\end{enumerate}
	Moreover, if these conditions are satisfied, then there exists a rational polyhedral fundamental domain for the action of $G$ on $C_{+}$.
\end{thm}

\subsection{The cone of curves for  surfaces with nef anticanonical divisor}

To study the structure of the effective cone of $\pr^3$ blown up at eight very general points, one crucial ingredient in the proofs of the results in Section \ref{section:effective_cone} is the understanding of the cone of curves of $\pr^2$ blown up at nine very general points.
We thank Artie Prendergast-Smith for communicating the following result to us:
\begin{prop}{\rm(\cite[Proposition 6.2]{Bor91})}\label{prop:surfaceclosedeffcone}
	Let $S$ be a complex projective smooth surface such that $-K_S$ is nef and non-torsion. Then $\NE(S)$ is a closed convex cone generated by the classes of rational curves and possibly the class $-K_S$, accumulating at most to $\R_{>0}[-K_S]$. 
	
	In particular, $\NE(S)\cap K_S^{\perp}$ is a rational polyhedral subcone of $\NE(S)$ generated by the classes of $(-2)$-curves, and possibly the class $-K_S$ if $(-K_S)^2=0$.
\end{prop}

\begin{rem}\label{rem:surfaceclosedeff}
	In Proposition \ref{prop:surfaceclosedeffcone}, if $S$ is isomorphic to the blowup $X_{2,9}$ of $\pr^2$ at nine very general points, then the unique element in $|{-}K_S|$ is a smooth elliptic curve and there is no $(-2)$-curve on the surface $S$. Moreover, by \cite[Corollary 1]{Dol83} there is a bijection between the set of $(-1)$-curves on $S$ and the orbit ${\cal W}_{2,9}\cdot e$, where $e$ is an exceptional curve of the blowup and ${\cal W}_{2,9}$ is the Weyl group of the Dynkin diagram $T_{2,6,3}$ acting on $N^1(X_{2,9})$. Hence $\Eff(S)$ is a closed convex cone generated by the classes in ${\cal W}_{2,9}\cdot e$ and the class $-K_S$. 
\end{rem}

\subsection{The Minimal Model Program}

In this subsection we recall some definitions in the context of the Minimal Model Program (MMP) and we prove a decomposition result for the movable effective cone in the setting of generalised pairs.

We first recall the following terminology and we refer to \cite[Sections 3.7 and 6.1]{KM98} for more details.
\begin{dfn}
     Let $(X,D)$ be a projective $\Q$-factorial pair.
     A \textit{flipping contraction} (resp.\ \textit{flopping contraction}) is an extremal birational contraction $f\colon X\to Y$ to a normal variety $Y$ such that the exceptional locus of $f$ has codimension at least two in $X$ and that $-K_X$ is $f$-ample (resp.\ $K_X$ is numerically $f$-trivial). 
    If $f\colon X \to Y$ is a flopping contraction such that $-(K_X+D)$ is $f$-ample, then the $(K_X+D)$-flip of $f$ is called the \textit{$D$-flop}.
\end{dfn}

We will need the following result for smooth threefolds whose anticanonical divisor is nef.

\begin{lem}\label{lem:two_cases_D}
 Let $X$ be a smooth projective threefold with $-K_X$ nef, and let $D$ be an effective $\Q$-divisor. There exists a finite sequence of flops $\varphi\colon X \dashrightarrow X'$ such that, denoting by $D'=\varphi_*(D)$, one of the following holds:
  \begin{enumerate}[leftmargin=*]
        \item $D'$ is nef,
        \item there exists a divisorial contraction $ X' \to Y$ 
    which contracts an irreducible component of $D'$.
    \end{enumerate}
    Moreover, case {\rm (i)} occurs if and only if $D\in \MovE(X)$.
 \end{lem}

 \begin{proof}

       Since $X$ is smooth, for sufficiently small $\epsilon>0$, the pair $(X,\epsilon D)$ is klt. Assume that $D$ is not nef, then by \cite[Lemma 2.3]{Xie24} there exists a $(K_X+\epsilon D)$-negative extremal ray $\Gamma$ such that $D\cdot \Gamma<0$. 
       Let $c_\Gamma$ be the contraction of the extremal ray $\Gamma$.

    If $c_\Gamma$ is small, then $K_X\cdot\Gamma =0$, since there are no
    flipping contractions for smooth threefolds. Hence, there exists a
    $D$-flop of $c_\Gamma$. Moreover, the flopped threefold $X^+$ is again smooth
    by \cite[Theorem 2.4]{Kol89} and $-K_{X^+}$ is nef. We  replace $X$ by $X^
    +$ and $D$ by its strict transform, and repeat the argument.  
    By the termination of three-dimensional flops \cite[Corollary
    6.19]{KM98}, this procedure terminates. Therefore, after finitely many flops
    $$\varphi\colon X\dashrightarrow X',$$
    with $D'\coloneqq \varphi_* D $, either $D'$ is nef, or the contraction
    $c_{\Gamma'}\colon X' \to Y$ of a $(K_{X'}+\epsilon D')$-negative extremal ray $\Gamma'$ satisfying $D'\cdot \Gamma' < 0$ is not small. For any curve $\ell$ contracted by $c_{\Gamma'}$, we have $D'\cdot \ell<0$ and thus $\ell\subset \Supp(D')$. Consequently,  $c_{\Gamma'}$ is a divisorial contraction which contracts an irreducible component of $D'$.

   The last statement follows from \cite[Lemma 2.5]{Dru11} together with 
\cite[Chapter III, Proposition 1.14]{Nak04}.
\end{proof}

Recall that a small $\Q$-factorial modification (SQM) of a normal $\Q$-factorial projective variety $X$ is an isomorphism in codimension one $X\dashrightarrow X'$ such that $X'$ is also $\Q$-factorial. Using techniques from the MMP, Kawamata proved in \cite{Kaw97} that for a normal $\Q$-factorial terminal Calabi-Yau threefold $X$, the following movable cone decomposition result holds:
\[
\MovE(X) = \bigcup_{\substack{\varphi\colon X\dashrightarrow X'\\ \text{SQM}}} \varphi^* \big(\NefE(X') \big),
\]
where $\{ \varphi\colon X\dashrightarrow X' \}$ ranges through all small $\Q$-factorial modifications of $X$. This result actually fits into a more general setting in terms of \emph{generalised pairs}.

Generalised pairs are introduced in \cite{BH14,BZ16}: very roughly speaking, they are couples $(X,B+M)$ such that $(X,B)$ is a usual pair and $M$ is the pushforward of a nef divisor from a higher birational model of $X$. We recall the following definitions and we refer to \cite{HL22,LT22} for further discussion.

\begin{dfn}\label{dfn:g-pairs}
	A \emph{generalised pair} or \emph{g-pair} $(X/Z,B+M)$ consists of a variety $X$ equipped with  projective morphisms 
	$$  X' \overset{f}{\longrightarrow} X \longrightarrow Z , $$ 
	where $ f $ is birational and $ X' $ is normal, an effective $ \R $-divisor $B$ on $X$, and an $\R$-Cartier $\R$-divisor $M'$ on $X'$ which is nef over $Z $ such that $ f_* M' = M $ and $ K_X + B + M $ is $ \R $-Cartier. If additionally $M'$ is an NQC divisor on $ X' $, then the g-pair $(X/Z,B+M)$ is an \emph{NQC g-pair}.	
\end{dfn}

For our purpose, we will always assume that the variety $Z$ is a point, in which case we denote the g-pair by $(X,B+M)$. We now define various classes of singularities of g-pairs.
\begin{dfn}
	Let $ (X/Z,B+M) $ be a g-pair with data $ X' \overset{f}{\to} X \to Z $ and $ M' $. We may write 
	\[ K_{X'} + B' + M' \sim_\R f^* ( K_X + B + M ) \]
	for some $ \R $-divisor $ B' $ on $ X' $. Let $E$ be a divisorial valuation over $ X $.  
    By replacing $X'$ with a higher birational model, we may assume that $E$ is a prime divisor on $X'$.
    The \emph{discrepancy of $ E $} with respect to $ (X,B+M) $ is \[ a (E, X, B+M) := {-} \text{mult}_E B'.\]
	We say that the g-pair $ (X,B+M) $ is: 
	\begin{enumerate}[leftmargin=*]
		\item \emph{klt} if $a (E, X, B+M) > -1 $ for all divisorial valuations $E$ over $X$,
		\item \emph{log canonical} if $a (E, X, B+M) \geq -1 $ for all divisorial valuations $E$ over $X$.
	\end{enumerate}
    \end{dfn}
    
    Note that if $M$ is a nef $\R$-Cartier $\R$-divisor, then the g-pair $(X,B+M)$ is klt (resp.\ log canonical) if and only if the usual pair $(X,B)$ is klt (resp.\ log canonical).

The following theorem is obtained in discussions with Vladimir Lazi\'{c}.
\begin{thm}\label{thm:g-mmp-movcone}
    Let $(X,B+M)$ be a projective $\Q$-factorial klt NQC g-pair of dimension $n$ such that $K_X+B+M\equiv 0$. 
    
    Assume the existence of minimal models for smooth varieties in dimension $n$. Then 
    \[
\MovP(X)\supset \MovE(X) = \bigcup_{\substack{\varphi\colon X\dashrightarrow X'\\ \text{SQM}}} \varphi^* \big(\NefE(X') \big),
\]
where $\{ \varphi\colon X\dashrightarrow X' \}$ ranges through all small $\Q$-factorial modifications of $X$.
The conclusion holds unconditionally if:
    \begin{enumerate}[leftmargin=*]
        \item $n\leq 4$, or
        \item $n=5$ and $X$ is uniruled.\footnote{By \cite[Lemma 3.18]{LMPTX}, $X$ is not uniruled if and only if $X$ is canonical, $B=0$ and $K_X\sim_{\Q} 0$.}
    \end{enumerate}
\end{thm}
\begin{proof}
{\em Step 1.}
Since the pullback of an effective nef divisor by an isomorphism in codimension one is an effective movable divisor, we obtain \[ \bigcup_{\substack{\varphi\colon X\dashrightarrow X'\\ \text{SQM}}}\varphi^* \big(\NefE(X') \big) \subset \MovE(X).\]

\smallskip

{\em Step 2.}
In this step
we show the following:
\begin{equation}\label{eq:aux1}
    \MovE(X) \subset \bigcup_{\substack{\varphi\colon X\dashrightarrow X'\\ \text{SQM}}} \varphi^* \big(\NefE(X') \big)
\end{equation}
and
\begin{equation}\label{eq:aux2}
    \MovE(X) \subset \MovP(X).
\end{equation}

    Let $D\in\MovE(X) $. By replacing $D$ with $\epsilon D$ for small enough $\epsilon >0$, we have that $\big(X,(B+D)+M\big)$ is an NQC log canonical generalised pair, and
    \begin{equation}\label{eq:g-pair}
        D\equiv K_X+(B+D)+M
    \end{equation}
as $K_X+B+M\equiv 0$. By \cite[Theorem 1.2 and Corollary 1.3]{LT22}, we may run a $\big(K_X+(B+D)+M\big)$-MMP with scaling of an ample divisor
    \[
    \varphi\colon \big(X,(B+D)+ M\big) \dashrightarrow \big(X',(B'+D')+ M'\big),
    \]
    where $B' = \varphi_*B, D'=\varphi_*D$, and the divisors $M$ and $M'$ are pushforwards of the same nef $\R$-divisors on a common birational model of $X$ and $X'$. Since $D$ is movable and by (\ref{eq:g-pair}), we infer that this MMP does not contract any divisor. Moreover, $\big(X',(B'+D')+ M'\big)$ is a $\Q$-factorial NQC log canonical generalised pair and together with (\ref{eq:g-pair}), we obtain that $D'$ is nef. 
    
    Note that $D'$ is also an effective divisor, since $D$ is effective and $\varphi$ is an isomorphism in codimension one. Thus we have $D=\varphi^* D'\in\varphi^*\big(\NefE(X')\big )$, which proves (\ref{eq:aux1}).

Applying \cite[Proposition 3.20]{HL22} to $K_{X'}+(B'+D')+ M'$, we have $D'\in\NefP(X')$. Thus we obtain $D=\varphi^* D'\in\MovP(X)$, which proves (\ref{eq:aux2}).

\smallskip
{\em Step 3.}
For $n\leq 4$, or for $n=5$ and $X$ is uniruled, we repeat verbatim {\em Step 2}, except that we invoke \cite[Corollary 1.4]{LT22} instead of \cite[Theorem 1.2 and Corollary 1.3]{LT22}.
\end{proof}

We point out here a crucial application of Theorem \ref{thm:g-mmp-movcone}: when $(X,B)$ is a projective $\Q$-factorial klt pair such that $-(K_X+B)$ is nef, one can take $M=-(K_X+B)$. Then $M$ is an NQC divisor on $X$ by \cite[Corollary 3.6]{HL21}, and $(X,B+M)$ is an NQC klt generalised pair. In particular, we obtain:
\begin{proof}[Proof of Proposition \ref{prop:main}]
    Setting $B = 0$ and $M=-K_X$, we see that $(X,B+M)$ is an NQC generalised pair. We apply Theorem \ref{thm:g-mmp-movcone} to the NQC generalised pair $\big(X,0+(-K_X)\big)$, taking into account that every $\varphi\colon X\dashrightarrow X'$ appearing in the union is a finite sequence of flops by the first two paragraphs of the proof of Lemma \ref{lem:two_cases_D}.
\end{proof}

\section{Blowup of $\pr^3$ at eight points}\label{sec:blowup}
Let $p_1,\dots,p_8$ be eight very general points in $\pr^3$ and let $\pi\colon X\to \pr^3$ be the blowup of $\pr^3$ at the points $p_1,\dots,p_8$. 

\medskip

From this point on, we use the following notation:
\begin{itemize}[leftmargin=*]
    \item $E_i\subset X$ is the exceptional divisor over $p_i\in\pr^3$, for $i=1,\dots,8$;

    \item $H\in \Pic(X)$ is the pullback $\pi^*\OO_{\pr^3}(1)$;

    \item $e_i$ is the class of line in $E_i$, for $i=1,\dots,8$;

    \item $h = H^2$ is the class of the transform of a general line in $\pr^3$;

    \item $\ell_{ij}$ is the class of the transform of the line through the two points $p_i$ and $p_j$, for $1\leq i<j \leq 8$.
\end{itemize}

We have $N^1(X) = \Z H \oplus \Z E_1 \oplus \dots \oplus \Z E_8$ and $ N_1(X) = \Z h \oplus \Z e_1 \oplus \dots \oplus \Z e_8$, and we will call these generators the \textit{standard bases} for $N^1(X)$ and $N_1(X)$, respectively. 

\begin{notation}
Let $d,m_1,\ldots,m_8 \in \R$. We denote by $(d;m_1,\dots,m_8)$ the class 
\[dH - \sum_{i=1}^8 m_i E_i \in N^1(X)_{\R}.\]
\end{notation}

In the standard basis, the anticanonical divisor class $-K_X$ is given by
\[
-K_X = 4H - 2\sum_{i=1}^8 E_i = (4;2,\ldots,2) = (4;2^8).
\]

As shown in \cite[Section 2]{LesOtt16}, $-K_X$ is nef and not semiample with $(-K_X)^3=0$. Moreover, a general member $Q$ in the linear system $|{-}\frac{1}{2}K_X|$ is smooth and isomorphic to $\pr^2$ blown up at nine very general points. Two general members in $|{-}\frac{1}{2}K_X|$ intersect transversally along a smooth elliptic curve which is the unique member in $|{-}K_Q|$.
Furthermore, a general member $\Delta\in|{-}K_X|$ is reducible and $(X,\Delta)$ forms a log canonical (or more precisely, a dlt) Calabi-Yau pair.

\subsection{The Weyl group}\label{subsec:weylgroup}
In this subsection, we introduce a Weyl group acting on $N^1(X)_{\R}$. For more detail we refer to \cite{Dol83} and \cite{Mukai04}. 
Consider the following symmetric bilinear form $(\cdot,\cdot)$ on $N^1(X)_\R$:
\begin{equation}\label{eq:pairing}
(H,H) = 2, \ (E_i,E_j)= -\delta_{i,j}, \ (H,E_j)= 0 
\end{equation}
for $ i,j \in\{1,\ldots, 8\}$. Note that by \cite[Theorem 2.10(1)]{DP19}, the form $(\cdot,\cdot)$ is preserved under $\cal W$.
\begin{lem}[\cite{Mukai04}]\label{lem:roots}
	The classes $\alpha_0,\ldots,\alpha_7, E_8$, where 
	\begin{align*}
	\alpha_0 &= H- E_1-\ldots-E_4, \\
	\alpha_1 &= E_1-E_2, \ldots, \alpha_7 = E_7-E_8,
	\end{align*}
	form another basis of $N^1(X)$. Moreover, $\alpha_0,\ldots,\alpha_7$ is a $\Z$-basis of the orthogonal complement $K_X^{\bot}$ of $K_X$ under $(\cdot,\cdot)$ and a system of simple roots of an affine root system with Dynkin diagram $T_{2,4,4}$. 
		\end{lem}

    We refer to \cite[Chapitre VI, \S 1 and \S 2]{Bou08} for basic properties of affine root systems and Weyl groups. We also point out that the Dynkin diagram $T_{2,4,4}$ is precisely the affine Dynkin diagram $\widetilde{E}_7$ in \cite[Chapitre VI, \S 4, Théorème 4]{Bou08}.

  Consider the real vector space $N^1(X)_{\R}$. Then $\{\alpha_0,\dots,\alpha_7\}$ is a set of linearly independent vectors in $N^1(X)_{\R}$.
   Let ${\cal W}$ be the Weyl group generated by the orthogonal reflections with respect to $\alpha_0,\ldots,\alpha_7$, that means, 
   \begin{equation}\label{eq:weylgroup}
    {\cal W} = \langle s_0,\ldots,s_7 \rangle,
   \end{equation}
    where $s_i(D) = D + (\alpha_i,D)\cdot \alpha_i$ for any $D\in N^1(X)_{\R}$.

The \textit{fundamental Weyl chamber} is defined as 
    \[
 C = \{ D \in N^1(X)_{\R} \mid (D,\alpha_i) > 0  \text{ for all } i\in\{0,\dots,7\} \},
 \]
  where the $\alpha_i$ are the simple roots of Lemma \ref{lem:roots} and $(\cdot,\cdot)$ the symmetric bilinear form from \eqref{eq:pairing}. 
  The closure $\overline{C}$ of $C$ is again a convex cone and characterized in the following way.
  Given any divisor class $D = dH - \sum_{i=1}^{8} m_iE_i$ in $N^1(X)_{\R}$, we have:
  \begin{align}\label{eq:fundchamber}
  D \in \overline{C} \quad  & \text{ if and only if }  (D,\alpha_i) \geq 0  \text{ for all } i\in\{0,\dots,7\} \nonumber \\ 
 &  \text{ if and only if } 2d \geq \sum_{i=1}^{4} m_i \text{ and } m_i \geq m_{i+1} \text{ for any }i\in\{1,\dots,7\}.
  \end{align}

\begin{dfn}(\cite{DeVolLaf07})
We say that a divisor class 
\[
D = (d;m_1,\ldots,m_8) = dH - \sum_{i=1}^{8} m_iE_i \in N^1(X)_{\R}
\]
 is in \textit{standard form} if $D \in \overline{C}$, that means $m_1 \geq m_2 \geq \ldots \geq m_8$ and $2d \geq \sum_{i=1}^{4} m_i$.
  \end{dfn}
 Note that $(K_X,\alpha_i) = 0$ for all $i\in\{0,\dots,7\}$, hence any $\R$-multiple of $K_X$ is in $\overline{C}$. The set $\overline{C}$ is thus determined modulo $K_X$ by the \textit{fundamental weights} $f_0,\ldots,f_7$ satisfying
 \[
 (f_i,\alpha_j) = \delta_{i,j}.
 \]
 More precisely, setting
 \begin{enumerate}[leftmargin=*]
 \item $f_0 = \frac{1}{2} H$,
 \item $f_1 = \frac{1}{2} H  - E_1$,
 \item $f_2 = H  - E_1 -E_2$,
  \item $f_3 = \frac{3}{2}H  - E_1 -E_2-E_3$,
  \item $f_4 = 2H  - \sum_{i=1}^4 E_i$,
   \item $f_5 = 2H - \sum_{i=1}^5 E_i$,
  \item $f_6 = 2H-\sum_{i=1}^6 E_i$,
   \item $f_7 = 2H - \sum_{i=1}^7 E_i$,
 \end{enumerate}
then any element in $\overline{C}$ is of the form 
 \[
 \sum_{i} \lambda_i f_i + \mu K_X
 \]
 with $\lambda_i \in \R_{\geq 0}$ and $\mu \in \R$.
 
 The union 
 \[
 T = \bigcup_{w \in {\cal W}} w\cdot\overline{C}
 \]
 is called the \textit{Tits cone}. It is a convex cone in $N^1(X)_{\R}$; see for example \cite[Chapitre V, \S 4, Proposition 6]{Bou08}.

    Now we analyse the action of the generators of ${\cal W}$ on the standard basis $\{H,E_1,\ldots,E_8\}$ of $N^1(X)_{\R}$. It is easy to see that for each $i\in\{1,\ldots,8\}$, the simple reflection $s_i$ permutes the elements $E_i$ and $E_{i+1}$ and leaves $E_j$ for $j \neq i,i+1$ and $H$ invariant. Moreover, for the reflection $s_0$ we obtain
  \begin{enumerate}[leftmargin=*]
  	\item $s_0(H) = 3H - 2\sum_{i=1}^{4} E_i$, 
  	\item $s_0(E_j) = H - \sum_{i=1,i\neq j}^{4} E_i$, for $j=1,\ldots,4$,
  	\item $s_0(E_j) = E_j$, else.
  \end{enumerate}
In the next subsection we will see that the simple reflection $s_0$ is exactly the automorphism $\phi_{\{1,\ldots,4\}}\colon N^1(X) \rightarrow N^1(X)$ induced by the standard Cremona transformation centred at the first four points $p_1,\ldots,p_4$. More generally, given an element $w\in{\cal W}$, there exists a corresponding isomorphism in codimension one for $X$:
\begin{thm}[\cite{Mukai04}]\label{thm:weylelem}
	For each element $w\colon N^1(X) \rightarrow N^1(X)$ of the Weyl group ${\cal W}$ there exists an isomorphism in codimension one  $\varphi_w\colon X \dashrightarrow X_w$, where $X_w$ is a blowup of $\pr^3$ at 8 very general points such thatthe pullback of the standard basis of $N^1(X_w)$ coincides with the transformation of that of $N^1(X)$ by $w$.
\end{thm}

\begin{rem}\label{rem:weylaction}
Let ${\cal V} = \big((\pr^3)^8\big)/\rm{PGL}(4)$, and consider $\textbf{p} = (p_1,\ldots,p_8) \in {\cal V}$.  Then for an element $w \in {\cal W}$, the variety $X_w$ from Theorem \ref{thm:weylelem} is the blowup of $\pr^3$ at $\textbf{q} = {\rm{cr}}_{3,8}(w)(\textbf{p})$, where 
\[
{\rm{cr}}_{3,8}\colon {\cal W} \rightarrow \rm{Bir}({\cal V})
\]
is the so called \textit{Coble representation} of ${\cal W}$ (see \cite{Mukai04}). 
Then, using \cite[Lemma 2.2]{lesieutre2014diminished},  there exists a set $U \subset {\cal V}$, the complement of a countable union of subvarieties of $ {\cal V}$ such that if $\textbf{p} \in U$, then for $X$ being the blowup of $\pr^3$ at $\textbf{p}$, the action of ${\cal W}$ on $N^1(X)_{\R}$ preserves the cones $\Eff(X)$ and $\MovC(X)$, respectively. 
\end{rem}

\subsection{The standard Cremona transformations}\label{subsec:cremona}

In this subsection, we describe the action of a standard Cremona transformation of $\pr^3$ centred at four points. Our presentation closely follows \cite[Section 3]{LesOtt16}, see also \cite[Section 3.1]{DimMir26} for more detail.

The standard Cremona transformation of $\pr^3$ (centred at the four coordinate points) is defined as
\[
\Cr \colon \pr^3 \dashrightarrow \pr^3, \quad [x_0:x_1:x_2:x_3] \mapsto \left[ \frac{1}{x_0}:\frac{1}{x_1}:\frac{1}{x_2}:\frac{1}{x_3}\right].
\]
Geometrically, $\Cr$ blows down the four planes, each of which passes through three of the coordinate points. Now let $\pi\colon Y\to \pr^3$ be the blowup of the four coordinate points. Then $\Cr$ can be lifted to a birational transformation $\overline{\Cr}\coloneqq \pi^{-1} \circ \Cr{} \circ \pi$ of $Y$, which does not contract any divisor and whose indeterminacy locus consists of the strict transforms of the six lines through pairs of the coordinate points.

Using the notation from Remark \ref{rem:weylaction}, we first consider $X$ as the blowup of $\pr^3$ at $\textbf{p} \in \mathcal{V}$, where $p_1,\ldots,p_4$ are the four coordinate points. Then we can embed $N^1(Y)$ into $N^1(X)$ with the complement spanned by $E_5,\ldots,E_8$. Thus, we can extend the isomorphism $\overline{\Cr}^*\colon N^1(Y)\to N^1(Y)$ to an isomorphism $N^1(X) \rightarrow N^1(X)$ by letting it act as the identity on this complement.

Now if $X$ is the blowup of very general $\textbf{p} = (p_1,\dots,p_8)\in \mathcal{V}$, then any choice of four of the eight points can be taken as the central points of a corresponding Cremona transformation (after a suitable linear transformation).
For any subset $I\subset\{1,\dots,8\}$ with $|I|=4$, we consider the Cremona transformation centred at $\{p_i\mid i\in I\}$. Then the induced action $\phi_I$ on $N^1(X)$ can be explicitly described as follows.
\begin{prop}{\rm (\cite[Proposition 3.1]{DimMir26}, \cite[Section 3.3]{BDP16})}\label{prop:crondivisors}
	
    Fix any subset $I\subset \{1,\dots,8\}$ with $|I|=4$. 
    \begin{enumerate}[leftmargin=*]
    \item The action of $\phi_I$ on $N^1(X)$ is given by sending $dH -\sum_{i=1}^8 m_i E_i $ to $d'H-\sum_{i=1}^8 m'_i E_i$, where
    \begin{align*}
        d' &= 3d - \sum_{j\in I} m_j = d + t,\\
        m'_i &= 2d - \sum_{j\in I, j\neq i} m_j= m_i+t \text{ for } i\in I \text{ and }
        m'_i = m_i \text{ for }i\not\in I,
    \end{align*}
    and $t\coloneqq 2d - \sum_{i\in I}m_i$.
    In particular, we have $\phi_I^2 = \Id$ and
     \begin{align*}
    \phi_I(H) &= 3H - 2\sum_{i\in I} E_i,\\
    \phi_I(E_i) &= H - \sum_{j\in I, j\neq i} E_j \text{ for } i\in I \text{ and } \phi_I(E_i)=E_i \text{ for } i\not\in I.
    \end{align*}
\item Let $D\subset X$ be an effective Cartier divisor. Then
\[
h^0\big ( X,\OO_X(D)\big) = h^0\big(X,\OO_X(\phi_I(D))\big).
\]
\end{enumerate}
\end{prop}
Consider an effective Cartier divisor $D\subset X$. Then every irreducible component of $D$ is either one of the exceptional divisors $E_i$, or the strict transform of a hypersurface in $\pr^3$. Hence, the divisor class of $D$ can be written as 
\[(d;m_1,\ldots,m_8) \text{ with } d\in\Z_{\geq 0} \text{ and } m_i\in\Z.\]
We call $d$ the \textit{degree} of $D$.  With notation from Proposition \ref{prop:crondivisors}, if $t=2d - \sum_{i\in I}m_i<0$ for some subset $I\subset\{1,\dots,8\}$ with $|I|=4$, then the degree of $D$ can be lowered by $\phi_I$. Moreover, by Proposition \ref{prop:crondivisors}(ii), the divisor $\phi_I(D)$ is again effective, and thus, its degree remains non-negative.

\subsection{Linear systems in standard form}
In this subsection we introduce the notion a linear system in standard form. The existence of this standard form (see Remark \ref{rem:standform}) together with the decomposition of a linear system in standard form (see Proposition \ref{prop:standform}) are some of the main tools to prove Theorems \ref{main:eff} and \ref{main:mov}.
\begin{dfn}(\cite{DeVolLaf07})
Fix $d\in \Z_{\geq 0}$ and $m_i\in\Z$ for $i\in\{1,\dots,8\}$. Let us denote by ${\cal L}(d;m_1,\dots,m_8)$ the linear system $|dH-\sum_{i=1}^8 m_i E_i|$. We say that ${\cal L}(d;m_1,\dots,m_8)$ is in \textit{standard form} if the class $(d;m_1,\ldots,m_8)$ is in standard form.
\end{dfn}

\begin{rem}\label{rem:standform}

Consider the class $D$ of an effective Cartier divisor in $X$. Then by permuting the order of the points and applying Proposition \ref{prop:crondivisors} iteratively (see \cite[Proposition 3.2]{LU10}), 
there exists an element $w \in \cal W$ and a class $(d;m_1,\ldots,m_8)$ in standard form  with $d \in \Z_{\geq 0}$ and $m_i \in \Z$ such that 
\[w\cdot D =(d;m_1,\ldots,m_8).\]
Moreover, by Proposition \ref{prop:crondivisors}(ii) we obtain \[\dim |D|=\dim {\cal L}(d;m_1,\ldots,m_8).\]
\end{rem}

\begin{notation}
	For a divisor class $(d;m_1,\ldots,m_8)$ (resp.\ a linear system ${\cal L}(d;m_1,\ldots,m_8)$) in standard form, we usually  omit those $m_i$ which are zero. For example, 
	we write 
	${\cal L}(1;1^2)$ for the system $|H-E_1-E_2|$.
\end{notation}

\begin{prop}{\rm(\cite[Proposition 2.2]{DeVolLaf07})}\label{prop:standform} Let $(d;m_1,\ldots,m_8)$ be a Cartier divisor class in standard form with $m_8 \geq 0$. Then 
\begin{equation}\label{eq:standform}
(d;m_1,\ldots,m_8)	= D'
+ \sum_{i=4}^{a-1} (m_i - m_{i+1})\cdot (2;1^i) + m_a \cdot (2;1^a), \end{equation}
	where $D' = 
	(d-2m_4;m_1-m_4,m_2-m_4,m_3-m_4)$ 
	and
	 $a = \max \{i \mid m_i \neq 0\}$.
	\end{prop}

\begin{rem}
	Note that every divisor class in \eqref{eq:standform} is in standard form and all the coefficients are non-negative. The divisor class $D'$ is effective if and only  if $d-2m_4 \geq m_1-m_4$ (see \cite[Lemma 5.1]{DeVolLaf07}). 
	\end{rem}
Writing a linear system in standard form allows us to determine its base locus, see \cite[Theorem 6.2]{DeVolLaf06}. For our purpose, we only need the following result, see \cite[Theorem 6.2]{DeVolLaf06} and \cite[Corollary 4.4]{BDP16}.
\begin{prop}\label{prop:baselocusstandardsystem}
	Let ${\cal L}(d;m_1,\dots,m_8)$ be a linear system in standard form. Assume that $m_8\geq 0$.
 Then ${\cal L}(d;m_1,\dots,m_8)$ has no fixed divisor.
\end{prop}
\begin{rem}\label{rem:standardmov}
Consider an effective Cartier divisor class $D$ such that $\Bs |D|$ has codimension at least $2$. By Remark \ref{rem:standform} there exists an element $w\in {\cal W}$ such that $w\cdot D = (d;m_1,\ldots,m_8)$ is in standard form, and we have $d\geq m_i$ for all $i\in\{1,\dots,8\}$ as $w\cdot D$ is again effective. 
Now assume that the class $(d;m_1;\ldots,m_8)$ has a negative coefficient $m_k$ for some $k\in\{1,\dots,8\}$, then the linear system ${\cal L}(d;m_1,\ldots,m_8)$ contains $E_k$ with multiplicity $-m_k$ in its base locus. But if this is the case, then there exists a fixed component with multiplicity $-m_k$ in the base locus of the starting linear system $|D|$; it is the image of $E_k$ by $w^{-1}$. This gives a contradiction.

In particular, if we start with a Cartier divisor class $D\in \Mov(X)$ in Remark \ref{rem:standform}, then there exists an element $w\in \cal W$ such that the class $ w\cdot D  = (d;m_1,\dots,m_8)$ is in standard form with $d\geq m_i\geq 0$ for all $i\in\{1,\dots,8\}$.
\end{rem}

\section{The cone of curves and the nef cone} Let $X$ be $\pr^3$ blown up at eight very general points $p_1,\ldots,p_8$. We use the notation introduced at the beginning of Section \ref{sec:blowup}.
In this section, we recall some results from \cite{CLO16}, which allow us to determine a set of generators for the cone of $\Nef(X)$ and to prove Theorem \ref{main:nef}.

Consider an irreducible curve $C$ in $X$. If $C$ is contained in some exceptional divisor $E_i$ for $i\in\{1,\dots,8\}$, then the class of $C$ is proportional to $e_i$. If $C$ is not contained in any exceptional divisor, then $C$ is the strict transform of a curve in $\mathbb{P}^3$ with positive degree $a$ and non-negative multiplicity $b_i\leq a$ at the point $p_i$ for $i\in\{1,\dots,8\}$. Since $2H - \sum_{i=1}^8 E_i$ is nef, we can apply \cite[Lemma 2.6]{CLO16} for $r=8$ and $n=2$ to obtain that the class of $C$ is a non-negative linear combination of $e_i$ for $1\leq i\leq 8$, and $\ell_{ij}$ for $1\leq i< j\leq 8$. Therefore, we have the following result:

\begin{prop}{\rm(\cite[Proposition 4.1]{CLO16})}\label{prop:coneofcurves}
    The cone $\NE(X)$ is a rational polyhedral 
cone spanned by the classes $e_i$, for $1\leq i\leq 8$, and the classes $\ell_{ij}$, for $1\leq i<j \leq 8$.
\end{prop}

In particular, Proposition \ref{prop:coneofcurves} implies that $\NE(X)$ is closed. We can now describe the nef cone $\Nef(X)$ which is dual to  $\overline{\NE}(X) = \NE(X)$.
\begin{lem}\label{lem:nefcone}
    The cone $\Nef(X)$ is a rational polyhedral cone spanned by the strict transforms of hyperplanes through at most one blown up point, and quadrics through at most eight blown up points in $\pr^3$. More precisely, $\Nef(X)$ is spanned by the classes 
    \begin{itemize}[leftmargin=*]
\item $H$,
\item $H-E_i$ for $i=1,\dots,8$,
\item $2H - \displaystyle\sum_{j\in J\subset\{1,\dots,8\},|J|=m} E_j$ for $3\leq m \leq 8$.
\end{itemize}
\end{lem}
\begin{proof}
	Consider a nef Cartier divisor class $D = dH - \sum_{i=1}^8 m_i E_i$.
	Then by Proposition \ref{prop:coneofcurves}, we have
	\[m_i\geq 0 \text{ for any } i=1,\dots,8 \text{ and } d\geq m_i+m_j \text{ for any } 1\leq i<j\leq 8.\]
		By reordering, we may assume that $m_1\geq m_2\geq \dots \geq m_8$.
	This implies that $D$ is in standard form. Similarly as in Proposition \ref{prop:standform}, we can write 
	\[
	D = (d-2m_2;m_1-m_2) + \sum_{i=2}^{a-1} (m_i - m_{i+1})\cdot (2;1^i) + m_a \cdot (2;1^a),
   \]
	where 
	$a = \max \{i \mid m_i \neq 0\}$. Any class of the form $(2;1^i)$ for $i \in \{2, \ldots,a\}$ is a $\Z_{\geq 0}$-linear combination of the classes in the lemma. 
	Furthermore, \[d-2m_2 \geq m_1-m_2 \geq 0,\] as $d \geq m_1+m_2$ and $m_1\geq m_2$. Thus,
	\[
	(d-2m_2;m_1-m_2) = (d-m_1-m_2) H + (m_1-m_2) (H-E_1),\] which finishes the proof. 
\end{proof}
Recall that  $\Aut(X)$ acts on $N^1(X)_{\R}$ by pulling back divisors and we denote by $\Aut(X)^*$ the image of $\Aut(X)$ in  $\GL(N^1(X)_{\R})$. 

As an immediate consequence of Lemma \ref{lem:nefcone},
the group  $\Aut(X)^*$ is finite (see e.g. \cite[Proposition 1.11]{Laz14}). 
We remark that  \cite[Theorem 1.1]{Gach25} shows that the pseudo-automorphism group $\PsAut(X)$ is trivial, and therefore, its subgroup 
$\Aut(X)$ is also trivial.

\begin{prop}\label{main:nef}
   Let $X$ be the blowup of $\pr^3$ at eight very general points. 
	Then the nef cone $\Nef(X)$ is rational polyhedral and there exists a rational polyhedral fundamental domain for the action of $\Aut(X)$ on $\NefE(X) = \Nef(X)$.
\end{prop}
\begin{proof}

By Lemma \ref{lem:nefcone}, we see that $\Nef(X)$ is a rational polyhedral cone and that every extremal ray of $\Nef(X)$ contains an effective class. Thus, 
\[
	\Nef(X) = \NefE(X) = \NefP(X).
	\]
Moreover, as $\Aut(X)^*$ is finite,
 this implies that the action of $\Aut(X)$ on $\NefE(X)$ admits a rational polyhedral fundamental domain by Theorem \ref{thm:looi} (or by \cite[Proposition 1.11]{Laz14}).
\end{proof}

\section{The effective cone}\label{section:effective_cone}
 As before, $X$ denotes the blowup of $\pr^3$ at eight  very general points. In this section, we explicitly describe a set of extremal rays for the convex cone $\Eff(X)$ and we prove Theorem \ref{main:eff}.

\begin{dfn}(\cite{Mukai01})
A divisor $D\subset X$ is a \textit{$(-1)$-divisor} if there exists an isomorphism in codimension one $f\colon X\dashrightarrow X'$ and a morphism $\pi'\colon X'\to Y$ that is the blowup of a projective variety $Y$ at a smooth point, such that $D$ is the strict transform of the exceptional divisor of $\pi'$.
\end{dfn}

In \cite[Lemma 3]{Mukai01}, Mukai proved that the class of a $(-1)$-divisor is contained in any set of generators of the monoid of effective divisor classes in $\Pic(X)$. Thus $(-1)$-divisors play an important role in understanding the cone $\Eff(X)$. Furthermore, Mukai proved in \cite[Theorem 1]{Mukai04} that any divisor in the orbit of the Weyl group of an exceptional divisor of $\pi\colon X\to \pr^3$ is a $(-1)$-divisor. To summarise, we have the following inclusions:

\begin{align*}
    \mathcal{W}\cdot E_8 &\subset \{ \text{classes of }(-1)\text{-divisors} \}\\
    &\subset \{ \text{generators of the monoid of the effective divisor classes}\}
\end{align*}

We first describe all the prime divisors on $X$ in the following lemma.
\begin{lem}\label{lem:primdiv}
    Let $D\subset X$ be a prime divisor on $X$. Then $D$ is either a $(-1)$-divisor or a movable divisor.
\end{lem}
\begin{proof}  
By Lemma \ref{lem:two_cases_D}, the prime divisor $D$ is either movable, or there exists a sequence of flops $\varphi\colon X\dashrightarrow X'$ with $D'\coloneqq \varphi_* D$ such that $D'$ is contracted by $c_{\Gamma'}$, where $c_{\Gamma'}$ corresponds to a $(K_{X'}+\epsilon D')$-negative extremal ray $\Gamma'$ with $D'\cdot \Gamma'<0$ for some small enough $\epsilon>0$. We will show that in the latter case, the divisor $D'$ is contracted to a smooth point by $c_{\Gamma'}$, which implies that $D$ is a $(-1)$-divisor.
    
    Indeed, if $\Gamma'$ is $K_{X'}$-negative, then by the classification of Mori-Mukai (see \cite[Section 3]{MM83}), we obtain that $D'\simeq \pr^2$  and $c_{\Gamma'}$ contracts $D'$ to a smooth point as $-K_{X'}$ is divisible by two in $\Pic(X')$.
    If $\Gamma'$ is $K_{X'}$-trivial, then $D'$ is covered by $K_{X'}$-trivial curves. Since $\varphi$ is a sequence of flops, we obtain that $D$ is covered by $K_X$-trivial curves.
    This implies that the set of $K_X$-trivial is uncountable, which contradicts \cite[Lemma 10]{LesOtt16} and proves the claim. 
\end{proof}

Now, the following lemma implies that the set of all $(-1)$-divisors on $X$ is equal to the set $\mathcal{W}\cdot E_8$.

\begin{lem}\label{lem:-1divinorbit}
    Let $D\subset X$ be a $(-1)$-divisor. Then $D$ is in the orbit of an exceptional divisor of $\pi\colon X\to\pr^3$ under the Weyl group.
\end{lem}
\begin{proof}
Let $(d;m_1,\dots,m_8)$ be the divisor class of $D$. Since $D$ is a $(-1)$-divisor, it is a prime divisor and we have $h^0\big ( X,\OO_X(D)\big ) =1$. 
By Remark \ref{rem:standform}, there exists $w\in \cal W$ such that $w\cdot D$ is in standard form and $$h^0\big ( X,\OO_X(w\cdot D)\big ) =h^0\big ( X,\OO_X(D)\big )=1$$ by Proposition \ref{prop:crondivisors}(ii). 
Moreover, $w\cdot D$ is a prime divisor by Theorem \ref{thm:weylelem} and Remark \ref{rem:weylaction}.
Thus by Proposition \ref{prop:baselocusstandardsystem}, we infer that either $w \cdot D=E_8$.
\end{proof}

Revisiting the chain of inclusion stated above, we have now
\begin{align*}
    \mathcal{W}\cdot E_8  &= \{ \text{classes of }(-1)\text{-divisors} \}\\
    &\subset \{ \text{generators of the monoid of the effective divisor classes}\}.
\end{align*}

We recall from the introduction that if $X_{n,k}$ denotes the blowup of $\pr^n$ in $k > n+1$ points in general position, then by \cite[Theorem 1.3]{CT06} the effective cone is rational polyhedral if and only if 
 \begin{equation}\label{eq_fingen}
 \frac{1}{n+1} + \frac{1}{k-n-1} > \frac{1}{2}.
 \end{equation}
 If \eqref{eq_fingen} is satisfied, then $\Eff(X_{n,k})$ is spanned be the finitely many elements in the orbit ${\cal W_{n,k}} \cdot E_k$, where ${\cal W_{n,k}}$ denotes the associated Weyl group. 
If $n=3$ and $k=8$, the Weyl group ${\cal W} =  {\cal W_{3,8}}$ is infinite, $\Eff(X)$ is not rational polyhedral, and we show that in this case the infinitely many divisors in the ${\cal W}$-orbit of $E_8$ span only a proper subcone of $\Eff(X)$.
We thank the referee for providing a simpler proof of the following lemma.
\begin{lem} The class $-\frac{1}{2}K_X\in\Eff(X)$ is not contained in the convex cone generated by the classes of ${\cal W} \cdot E_8$. 
\end{lem}
\begin{proof}
Let $D\in \cal W\cdot E_8$. Since $-K_X$ is fixed by $\cal W$ and since the form $(\cdot,\cdot)$ defined in \eqref{eq:pairing} is preserved under $\cal W$, we obtain
\begin{equation}\label{eq:pairing_-1-divisor}
\left( -\frac{1}{2}K_X, D\right) = \left(-\frac{1}{2}K_X, E_8\right) =1.
\end{equation}
But $\left( -\frac{1}{2}K_X, -\frac{1}{2}K_X\right)=0$, which shows that $-\frac{1}{2}K_X$ cannot be a non-negative combination of elements in $\cal W\cdot E_8$.
\end{proof}

We can now describe the extremal rays generating the cone $\Eff(X)$ as a convex cone. Using the formula from Proposition \ref{prop:crondivisors}(i), we compute the elements in ${\cal W}\cdot E_8$ of degree at most $2$, displaying them up to permutations of the last eight entries:
\begin{center}
\begin{table}[h]
\begin{tabular}{ccccccccc}
	$H$ & $E_1$ & $E_2$ & $E_3$ & $E_4$ & $E_5$ & $E_6$ & $E_7$ & $E_8$ \\ \hline \hline
	$0$ & $1$ & $0$ & $0$ & $0$ & $0$ & $0$ & $0$ & $0$ \\
    $1$ & $-1$ & $-1$ & $-1$ & $0$ & $0$ & $0$ & $0$ & $0$ \\
	$2$ & $-2$ & $-1$ & $-1$ & $-1$ & $-1$ & $-1$ & $0$ & $0$ \\
\end{tabular}
\caption{$(-1)$-divisors up to degree $2$}
\label{table:fixeddiv}
\end{table}
\end{center}
\vspace{-2em}
where the entries in the table denote the coefficient of the corresponding divisor. For example, the second row represents the class $H -E_1-E_2-E_3$.

\begin{thm} \label{thm:effcone}
	The effective cone $\Eff(X)$ is the convex cone generated by the classes in ${\cal W} \cdot E_8$ and the class of $-\frac{1}{2}K_X$.
	\end{thm}
\begin{proof}
Let ${\cal C}$ be the convex cone generated by the classes in ${\cal W} \cdot E_8$ and the class $-\frac{1}{2}K_X$.  Then $\cal C \subseteq \Eff(X)$.
To prove the reverse inclusion, recall that $\Eff(X)$ is the convex cone generated by the classes of effective Cartier divisors. 
Since every effective Cartier divisor is a finite non-negative  combination of prime divisors, it suffices to show that the class of any prime divisor belongs to ${\cal C}$.
Moreover, since $-K_X$ is fixed by $\cal W$,
the cone ${\cal C}$ is preserved by ${\cal W}$. Hence, for any divisor class $D\in \Eff(X)$, we have $D \in {\cal C}$ if and only  $w\cdot D \in {\cal C}$ for some  $w\in{\cal W}$. Therefore, it suffices to prove that for every prime divisor class $D$, there exists an element $w \in {\cal W}$ such that $w\cdot D \in {\cal C}.$ 

	Let $D$ be a prime divisor class. By Remark \ref{rem:standform}, there exists an element $ w \in {\cal W}$ such that
	\[ w \cdot D = dH - \sum_{i=1}^8 m_i E_i\]
is an effective divisor class, where
\begin{equation}\label{eq:standard1}
 2d \geq \sum_{i=1}^4 m_i   
\end{equation}
and
\begin{equation}\label{eq:standard2}
    m_1 \geq \ldots \geq m_8,
\end{equation} 
i.e.\ $w\cdot D$ is in standard form. Moreover, $w\cdot D$ is a prime divisor by Theorem \ref{thm:weylelem} and Remark \ref{rem:weylaction}. Then either $w\cdot D = E_8 \in {\cal C}$, and we are finished, or $w\cdot D$ is the strict transform of a hypersurface in $\pr^3$
and we have
\begin{equation}\label{eq:primedivisor2}
    d \geq m_i \geq 0, \quad \forall i\in\{1,\dots,8\}.
\end{equation}
In the remainder of the proof, we will show that $w\cdot D\in {\cal C}$.

\smallskip
{\em Step 1.} We first consider the case when $d\geq m_1 + m_4$.

	  Applying Proposition \ref{prop:standform} we have 
\[
	 (d;m_1,\ldots,m_8) = D'
	  + \sum_{i=4}^{8} c_i \cdot (2;1^i),
	\]
	where $ D' = (d-2m_4;m_1-m_4,m_2-m_4,m_3-m_4)$ and 
	$c_i \geq 0$. 
	
	First we remark that any $2H - \sum_{i=1}^{k} E_i$ for $4 \leq k \leq 8$ is contained in ${\cal C}$. 
	Indeed, using the list of elements in ${\cal W} \cdot E_8$ up to degree $2$ in Table \ref{table:fixeddiv}, we see that for any $4 \leq k \leq 6$, the class of $2H - \sum_{i=1}^{k} E_i$ is a non-negative combination of classes in ${\cal W} \cdot E_8$, whereas $2H - \sum_{i=1}^{8} E_i$ is the class of $-\frac{1}{2}K_X$ and $2H - \sum_{i=1}^{7} E_i = 
	(2H - \sum_{i=1}^{8} E_i) + E_8 \in {\cal C}$.

	Consequently, it remains to show that the class $D'$ is contained in ${\cal C}$. 
	We write $D' = d'H -  \sum_{i=1}^{3} m_i'E_i$. Then $2d' \geq m_1'+m_2'+m_3'$ by (\ref{eq:standard1}) and $m_1' \geq m_2' \geq m_3'\geq 0$ by (\ref{eq:standard2}). Since $d\geq m_1+m_4$, we also have $d'\geq m_i'$ for all $i\in\{1,2,3\}$. Then we have 
    \[
    (d';m'_1,m'_2,m'_3) = d'(1;1^3) + \sum_{i=1}^3 (d'-m_i')E_i,
    \]
    
	where the right hand side is contained in ${\cal C}$. 
	Thus, we have shown that 
	\[
	 w \cdot D \in {\cal C}. 
	\]

\smallskip
 {\em Step 2.} Now we consider the case when $d<m_1 + m_4$. The proof of this step is inspired by the proof of \cite[Theorem 6.2]{DeVolLaf06}.
Let $t_{i,j}\coloneqq m_i+m_j-d$ for $1\leq i<j \leq 8$. Then $d<m_1 + m_4$ is equivalent to 
\begin{equation}\label{eq:case2}
    t_{1,4}>0.
\end{equation}
Therefore, 
\begin{equation}\label{eq:m1andm2}
    m_1\geq m_2+2 \geq 3.
\end{equation}
Indeed, we have
\begin{align*}
m_1 &\geq m_1 + \left(\sum_{i=1}^4 m_i\right) - 2d & \text{ by (\ref{eq:standard1})}\\
&= m_2 + (m_3 - m_4) + 2t_{1,4} & \\
&\geq m_2 + 2 & \text{ by (\ref{eq:standard2})(\ref{eq:case2}).}
\end{align*}
Let $a\coloneqq\max \{i\in\{1,\dots,8\}\mid m_i\neq 0\}$. Since $m_1\leq d < m_1+m_4$, we have $m_4>0$ and thus $a\geq 4$.

Note that using the list of elements in ${\cal W} \cdot E_8$ up to degree $2$ in Table \ref{table:fixeddiv},
we have 
\[
(3;3,1^7)= (2;2,1^5) + (H - E_1 - E_7 - E_8)\in {\cal C},
\]
and
$(3;3,1^{a-1})= (3;3,1^7) + \sum_{i=a+1}^8 E_i \in {\cal C}$ for $4\leq a<8$.

Now consider
\[
(d';m'_1,\dots,m'_{a'})\coloneqq (d;m_1,\dots,m_a) - (3;3,1^{a-1}).
\]
Note that by (\ref{eq:primedivisor2}) and (\ref{eq:m1andm2}), the case $d <  m_1 + m_4$ can only occur if $d \geq m_1 \geq 3$.
Then
\begin{align*}
d'&=d-3, \\
m'_1 &= m_1-3, \\
m'_i &= m_i-1, \quad \forall i\in \{2,\dots,a\},\\
a'&\coloneqq \max\{ i\in\{1,\dots,8\}\mid m'_i\neq 0\} \leq a.
\end{align*}
By construction, we have $m'_i= 0$ for $a'<i\leq 8$ and 
\[
d'\geq m'_1\geq m'_2\geq \dots \geq m'_{a'}>0
\]
 by (\ref{eq:standard2}) and (\ref{eq:m1andm2}).
Moreover, by (\ref{eq:standard1}) we have \[2d'= 2d-6 \geq \sum_{i=1}^4 m_i - 6 = \sum_{i=1}^4 m'_i.\]
Consequently, $(d';m'_1,\dots,m'_{a'})$ is in standard form with $d'\geq m'_i\geq 0$ for any $i\in\{1,\dots,a'\}$.

Furthermore, set $t'_{i,j}\coloneqq m'_i+m'_j-d'$ for $1\leq i<j \leq 8$. Then we have $t'_{1,4} = t_{1,4}-1$. If $t'_{1,4}>0$, since $(d';m'_1,\dots,m'_{a'})$ is in standard form, we can start our above procedure in {\em Step 2} again and do this until $t'_{1,4}=0$ for some $(d';m'_1,\dots,m'_{a'})$. We are now under the assumption in {\em Step 1} and thus $(d';m'_1,\dots,m'_{a'})\in {\cal C}$.
\end{proof}

In the remainder of the section, we will show that the cone $\Eff(X)$ is closed. We first need the following lemma to show that any $(-1)$-divisor on $X$ restricts to a $(-1)$-curve on a general member of $|{-}\frac{1}{2}K_X|$.
\begin{lem}\label{lem:intersectnum}
	Let $D$ be a $(-1)$-divisor on $X$ and let $Q\in |{-}\frac{1}{2}K_X|$. Then we have $D^2\cdot Q = -1$ and $D\cdot Q^2=1$.
\end{lem}
\begin{proof}
	Recall that $Q\sim 2H - \sum_{i=1}^8 E_i$ and we may write 
	\[
	D\sim dH - \sum_{i=1}^8 m_i E_i
	\]
	with $d,m_1,\dots m_8\in\Z$. 
	
	Since $H^3=1, H^2\cdot E_i = H\cdot E_i^2 =0, E_i^3 =1$ and $E_i\cdot E_j =0$ for $1\leq i < j\leq 8$, we obtain
	\begin{equation}\label{eq:intersect1}
	D^2\cdot Q = \left(dH - \sum_{i=1}^8 m_i E_i\right)^2\cdot \left(2H - \sum_{i=1}^8 E_i\right) = 2d^2 -\sum_{i=1}^8 m_i^2,
	\end{equation}
	\begin{equation}\label{eq:intersect2}
	D\cdot Q^2 = \left(dH - \sum_{i=1}^8 m_i E_i\right)\cdot \left(2H - \sum_{i=1}^8 E_i\right)^2 = 4d - \sum_{i=1}^8 m_i.
	\end{equation}
	
	Consider the symmetric bilinear form $(\cdot,\cdot)$ on $N^1(X)$ defined in (\ref{eq:pairing}) and we have
	\begin{equation}\label{eq:intersectpairing1}
	(D,D) = \left(dH - \sum_{i=1}^8 m_i E_i, dH - \sum_{i=1}^8 m_i E_i\right) = 2d^2 -\sum_{i=1}^8 m_i^2 = D^2\cdot Q
	\end{equation}
	by (\ref{eq:intersect1}), and
	\begin{equation}\label{eq:intersectpairing2}
	(D,Q) = \left(dH - \sum_{i=1}^8 m_i E_i, 2H - \sum_{i=1}^8 E_i\right) = 4d-\sum_{i=1}^8 m_i= D\cdot Q^2
	\end{equation}
	by (\ref{eq:intersect2}).
    Hence, by Lemma \ref{lem:-1divinorbit} we obtain
	\[
	(D,D) = (E_8,E_8) = -1 \text{ and } (D,Q) = (E_8,Q) = 1.
	\]
	Together with (\ref{eq:intersectpairing1}) and (\ref{eq:intersectpairing2}), we conclude $D^2\cdot Q=-1$ and $D\cdot Q^2=1$.
\end{proof}

\begin{lem}\label{lem:restrictfixdiv}
	Let $Q\in|{-}\frac{1}{2}K_X|$ be a general member. Then
	we have an injection:
	\begin{align*}
	\{(-1)\text{-divisors on }X\}&\hookrightarrow\{(-1)\text{-curves on }Q\} \\
	D &\mapsto D|_Q
	\end{align*}

\end{lem}
\begin{proof}
	By the adjunction formula, we have \[-K_Q \sim (-K_X-Q)|_Q \sim Q|_Q.\]
	
	Let $D$ be a $(-1)$-divisor on $X$ and denote by $\Gamma$ the curve $D\cap Q$. On the surface $Q$, we have
	\begin{equation}\label{eq:-1curve}
	-K_Q\cdot \Gamma = Q|_Q\cdot D|_Q =1 \text{ and } \Gamma^2 = (D|_Q)^2=-1,
	\end{equation}
	since $D\cdot Q^2=1$ and $D^2\cdot Q=-1$ by Lemma \ref{lem:intersectnum}.
	
	Since $D$ is a $(-1)$-divisor, by the proof of Lemma \ref{lem:primdiv} there exists a finite sequence of flops $X\dashrightarrow X'$ such that the strict transform $D'$ of $D$ on $X'$ is isomorphic to $\pr^2$. Denote by $Q'$ the strict transform of $Q$ on $X'$. As $\cal{O}_{D'}(D') = \cal{O}_{\pr^2}(-1)$, we have that $|Q'|_{D'} \subset |Q'|_{D'}|$  is a pencil of smooth rational curves on $D'\simeq \pr^2$. Hence $|Q|_{D}$ is a pencil of integral rational curves;  in particular, $\Gamma$ is an integral rational curve. Together with (\ref{eq:-1curve}), we obtain that $\Gamma$ is a $(-1)$-curve on the surface $Q$; in particular, $\Gamma$ is a smooth rational curve. Thus $D|_Q$ is a $(-1)$-curve on the surface $Q$ and we obtain a well-defined map:
	\begin{align*}
	\{(-1)\text{-divisors on }X\}&\to \{(-1)\text{-curves on }Q\}\\
	D &\mapsto D|_Q.
	\end{align*}
	By the short exact sequence
	\[
	0\to {\cal O}_X(D-Q) \to {\cal O}_X(D)\to {\cal O}_Q(D)\to 0
	\]
	and since $H^0\big( X,{\cal O}_X(D-Q)\big)=0$, the map defined above is injective.
\end{proof}

In the following,  we will work with the real projective space 
$\R\pr^8 = \big( N^1(X)_{\R} \backslash \{ 0 \} \big) / \R^*$. For any class $x\in \Psef(X)$ seen as a non-zero vector in $ N^1(X)_{\R} \simeq \R^9$,
the ray spanned by the class $x$ intersects the unit sphere  $\mathbb{S}^8 \subset \R^9$ in a unique point. Thus, we can identify each ray with a point in the projective space $\R\pr^8$.
We denote by $x=(d;m_1,\ldots,m_8)$ an element in $\Psef(X) \backslash \{ 0 \}$ and by ${\bf x} = [d:m_1:\ldots:m_8]$ its corresponding point in $\R\pr^8$. 

\begin{prop}\label{prop:accpoint}
	 Let $A = \{w\cdot E_8 \mid w \in {\cal W}\}$. Then the only accumulation point of the set $\{{\bf x}\}_{x \in A} \subset \R\pr^8$ is the point corresponding to the class of $-\frac{1}{2}K_X$.
     In particular, denoting by $\widetilde{{\bf x}}$ the point of $\R\pr^8$
   corresponding to the ray spanned by the class of $-\frac{1}{2}K_X$, the set $\widetilde{A}\coloneqq \{ {\bf x} \}_{x \in A}\cup \{ \widetilde{{\bf x}} \}\subset \R\pr^8$ is closed.
\end{prop}

\begin{proof}
	Fix a very general member $Q\in |{-}\frac{1}{2}K_X|$. 
	Now since $Q$ is the blowup of a smooth quadric surface in $\pr^3$ at eight very general points, we can write 
	\[N^1(Q)=\Z \ell_1 \oplus \Z \ell_2 \oplus \Z e_1\oplus\dots\oplus \Z e_8 \cong \Z^{10},\]
	where $\ell_1$ and $\ell_2$ are respectively the pullbacks of the two rulings of the quadric surface in $\pr^3$. Analogously as for $X$, we identify the rays in $N^1(Q)_{\R}$ spanned by $(-1)$-curves and $-K_Q$ with their images in $\R\pr^9$.
	We have
	\begin{equation}\label{eq:restricttoQ}
	H|_Q\sim \ell_1+\ell_2 \text{ and } E_i|_Q\sim e_i  \text{ for } i\in\{1,\dots,8\}.
	\end{equation}

	By compactness of $\R\pr^8$, there exists an accumulation point of the set $\{ {\bf x}\}_{x\in A}\subset \R\pr^8$. In the following, we show that there is exactly one accumulation point. 
	
	Let $\{ {\bf x}_k\}_{k\in\N}$ be a converging subsequence of $\{ {\bf x}\}_{x\in A}$ with limit point ${\bf y}\in\R\pr^8$, let $D_k$ be a $(-1)$-divisor on $X$ corresponding to the class $x_k$ for each $k\in \Z_{\geq 0}$ and let $B$ be an $\R$-divisor on $X$ corresponding to the class $y$. We can write 
	\begin{equation}\label{eq:limit}
	B\sim aH - \sum_{i=1}^8 b_i E_i,  
	\end{equation}
	where $a,b_i\in\R$ for any $i\in\{1,\dots,8\}$.

	By (\ref{eq:restricttoQ}), the map in Lemma \ref{lem:restrictfixdiv} can be written as:
	\begin{align*}
	\{(-1)\text{-divisors on }X\}&\hookrightarrow\{(-1)\text{-curves on }Q\} \\
	D  \sim dH - \sum_{i=1}^8 m_i E_i &\mapsto D|_Q \sim d(\ell_1+\ell_2) - \sum_{i=1}^8 m_i e_i.
	\end{align*} 
Thus, transferring to $\R\pr^8$ respectively $\R\pr^9$ as above, this injection is the restriction of the embedding 
	\begin{align*}
	\R\pr^8 &\hookrightarrow\R\pr^9 \\
	[d:m_1:\dots:m_8] &\mapsto [d:d:m_1:\dots :m_8].
	\end{align*}
 The embedding is induced by a linear map of Euclidean spaces and is therefore continuous with respect to the quotient topology on $\R\pr^8$ and $\R\pr^9$, respectively.
	Since the sequence $\{ {\bf x}_k\}_{k\in\N}$ converges to the point ${\bf y}$, we obtain that the sequence $\{ D_k|_Q\}_{k\in\N}$ converges to $B|_Q$ in $\R\pr^9$. Since $Q$ is isomorphic to $\pr^2$ blown up at nine very general points, by Proposition \ref{prop:surfaceclosedeffcone} and Remark \ref{rem:surfaceclosedeff} we obtain $[B|_Q]\in\R_{>0}[-K_Q]$, and thus there exists $\lambda\in\R_{>0}$ such that 
	\[
	B|_Q \sim \lambda (2\ell_1+2\ell_2-\sum_{i=1}^8 e_i).
	\]
	Hence, by (\ref{eq:limit}) we have $a= 2\lambda$ and $b_i = \lambda$ for any $i\in\{1,\dots,8\}$ and thus the limit point $\textbf{y}$ corresponds to the ray spanned by the class $-\frac{1}{2}K_X$.
    The last statement follows immediately from the first statement.
\end{proof}

We would like to thank Artie Prendergast-Smith for helpful discussions leading to the following statement,
and the referee for suggesting an alternative proof.

\begin{prop}\label{prop:effconeclosed}
	The effective cone $\Eff(X)$ is closed.
\end{prop}
\begin{proof}
By Proposition \ref {prop:accpoint} and with the same notation therein, the set $\widetilde{A}$ is closed, and hence compact in $\R\pr^8$.
Since $\Psef(X)$ is a closed pointed cone, its dual cone $\Psef(X)^\vee$ is full-dimensional, and hence, has a non-empty interior. 
We pick $\ell\in \operatorname{Int}(\Psef(X)^\vee)$ and we define
\begin{align*}
    \varphi\colon N^1(X)_{\R} &\to \R\\
    x&\mapsto x\cdot \ell.
\end{align*}
Then
\begin{equation}\label{eq:pos_pairing}
\varphi(x) > 0 \text{ for all }x \in \Psef(X)\backslash\{0\}.
\end{equation}
Let
\[
L\coloneqq\{x\in N^1(X)_{\mathbb R}\mid \varphi(x)=1\}
\]
and let
\[
C\coloneqq \Psef(X)\cap L.
\]
Then $C$ is a convex compact subset of $N^1(X)_{\R}$.

By \eqref{eq:pos_pairing}, every ray in 
$ A\cup \R_{\geq 0}\left[-\frac{1}{2}K_X\right]=\cup_{w\in \cal{W}}\R_{\geq 0}[w\cdot E_8] \ \cup \R_{\geq 0}\left[-\frac{1}{2}K_X\right]$
meets $L$ at a unique point; let us denote by $S$ its intersection with $L$.
Note that the map
 \begin{align*}
     \alpha\colon \Psef(X)\backslash \{0\} &\to L\\ 
       x  &\mapsto \frac{x}{\varphi(x)}
\end{align*}
   is continuous and constant on any positive ray, thus $\alpha$ factors through the projectivisation of $\Psef(X)\backslash \{0\}$.
   Since $S$ is the image of the compact set $\widetilde{A}$ under the induced continuous map, $S$ is compact in $L$, and thus also in $N^1(X)_{\R}$.

By Theorem \ref{thm:effcone},
we have
\[
\Eff(X)\cap L=\operatorname{conv}(S).
\]
Since $S$ is compact in $N^1(X)_{\R}$, it follows from Carathéodory's theorem that $\operatorname{conv}(S)$ is also compact. In particular, $\operatorname{conv}(S)$ is closed and it remains to show that $\Eff(X) = \R_{\geq 0} \cdot \mathrm{conv}(S)$ is closed.

Consider a sequence $y_n \coloneqq \lambda_n x_n$ with $\lambda_n \in \R_{\geq 0}$ and $x_n \in \mathrm{conv}(S)\subset L$ such that $\lim_{n \to \infty} y_n = y\in N^1(X)_{\R}$. Since $\varphi$ is linear and continuous, we infer that
\[
\varphi(y_n) = \varphi(\lambda_n x_n) = \lambda_n \varphi(x_n) = \lambda_n \geq 0
\]
converges to $\lambda\coloneqq \varphi(y) \geq 0$. 
Moreover, since $\conv(S)$ is compact, after passing to a convergent subsequence, we may assume $\lim_{n \to \infty} x_n = x \in \conv(S).$
Consequently, 
\[
y = \lim_{n \to \infty} y_n = \lim_{n \to \infty} \lambda_n x_n = \lambda x \in \R_{\geq 0} \cdot \mathrm{conv}(S),
\]
which finishes the proof. 
\end{proof}

\begin{proof}[Proof of Theorem \ref{main:eff}]
    It follows from Theorem \ref{thm:effcone} and Proposition \ref{prop:effconeclosed}.
\end{proof}

\section{The movable cone} 
We recall that $X$ denotes the blowup of $\pr^3$ at eight very general points. In this section we prove Theorem \ref{main:mov}.  The first part of Theorem \ref{main:mov}(ii) is a direct consequence of the fact that the infinite Weyl group preserves the effective movable cone $\MovE(X) =   \MovC(X) \cap \Eff(X)$:

\begin{lem}\label{lem:notratpol}
    The cone $\MovE(X)$ is not rational polyhedral.
\end{lem}
\begin{proof}
    Suppose by contradiction that $\MovE(X)$ is a rational polyhedral cone spanned by finitely many extremal rays and denote by $D_1,\dots, D_k$ the primitive classes generating these extremal rays, i.e.\ each $D_i$ is the smallest Cartier divisor class generating an extremal ray. Since the Weyl group ${\cal W}$ preserves the cone $\MovE(X)$, any $w\in{\cal W}$ permutes the classes $D_1,\dots, D_k$. Hence the class $\sum_{i=1}^k D_i$ is fixed by the group ${\cal W}$. Moreover, $\sum_{i=1}^k D_i$ is contained in the interior of $\MovE(X)$ and thus is big. But any class fixed by ${\cal W}$ is proportional to the class of $-\frac{1}{2}K_X$ which is not big, we obtain a contradiction.
\end{proof}

The proof of the remaining part of Theorem \ref{main:mov} is more involved and we split it in several steps. First, we show that we can cover the rational points of $\Mov(X)$ by ${\cal W}$-translates of a rational polyhedral cone $\Pi$, see Theorem \ref{thm:coverMov}. Then, to deduce that the whole cone $\Mov(X)$ is contained in ${\cal W} \cdot \Pi  \coloneqq \bigcup_{w \in {\cal W}} w \cdot \Pi$, we prove that  ${\cal W} \cdot \Pi$ is convex, see Theorem \ref{thm:wPiconvex}. Finally, the existence of a fundamental domain for the action of ${\cal W}$ on $\MovE(X)$ follows from Theorem \ref{thm:looi}.
\begin{dfn}\label{def:conePi}

    Throughout this chapter, we denote by $\Pi$ the rational polyhedral cone in $N^1(X)_{\R}$ spanned by the classes of 
 \begin{itemize}[leftmargin=*]
	\item $H$, 
 \item $H-E_i$ for $1 \leq i \leq 3$,  
 \item $H- E_i - E_j$ for  $1 \leq i < j \leq 3$, 
 \item $L_k = 2H - \sum_{i=1}^{k} E_i$ for $4 \leq k \leq 8$, 
 \item $M_\ell = 3H -3E_1 - \sum_{i=2}^\ell E_i$ for $4\leq \ell\leq 8$. 
 \end{itemize} 
\end{dfn}
The above generators have the following geometric interpretation. The first two types of classes are nef by Lemma \ref{lem:nefcone}. The third type of classes are not nef: each class has negative intersection with the strict transform of the line through $p_i,p_j \in \pr^3$. The class $L_k$ is nef by Lemma \ref{lem:nefcone} and corresponds to the strict transform of a quadric through $k$ blown up points in $\pr^3$. The class $M_\ell$ corresponds to the strict transform of a cone with vertex $p_1$ over a plane cubic through $\ell-1$ points, where the $\ell-1$ points are the projection of $p_2,\dots, p_\ell$ from $p_1$ to the plane. Note that $M_\ell$ is not nef, as it has negative intersection with the strict transform of the line through $p_1, p_i\in \pr^3$, where $i\in \{ 2,\dots, \ell\}$.

\begin{lem}\label{lem:standardinPi}
   Let $\Pi$ be as above, and  
 let $D=(d;m_1,\dots,m_8)$ be a Cartier divisor class in standard form such that $d\geq m_i\geq 0$ for all $i\in\{1,\dots,8\}$. Then $D\in\Pi$. 
\end{lem}

\begin{proof}
Let $D=(d;m_1,\dots,m_8)$ be a Cartier divisor class in standard form such that $d\geq m_i\geq 0$ for all $i\in\{1,\dots,8\}$. Then $2d\geq \sum_{i=1}^4 m_i$ and $d\geq m_1\geq\dots\geq m_8\geq 0$.

\smallskip
    {\em Step 1.} We first consider the case when $d\geq m_1+m_4$.
 
 By Proposition \ref{prop:standform}, we can write
	\[
D=(d-2m_4;m_1-m_4,m_2-m_4,m_3-m_4)
+ \sum_{i=4}^{8} c_i \cdot (2;1^i)
\]
with $c_i \geq 0$.

	 Since every divisor class $(2;1^i)$ for $i \geq 4$ is contained in $\Pi$, it is enough to show that for any divisor class $D$  of the form:
	\[
	D = (d; m_1,m_2,m_3),
	\]
	where $2d \geq m_1+m_2+m_3$ and $d\geq m_1\geq m_2 \geq m_3\geq 0$, we have $D \in \Pi$. Note that we can assume $d >0$, as $d =0$ implies that $D = 0$.

	We show the following stronger statement by induction on $d$:
	
 {\em Claim.} Any class 
	\[
	D = (d; m_1,m_2,m_3)
	\]
  with $d>0$, $2d \geq m_1+m_2+m_3$  and  $0 \leq m_i \leq d$ is contained in the cone $\Pi$.
	
\smallskip
	For $d =1$, the only possible classes are exactly the generators of $\Pi$ of degree 1. 
	Now let $d > 1$. We distinguish the following cases.
	\begin{itemize}[leftmargin=*]
	\item If all $m_i$ are zero, then $D$ is a multiple of $H$ which is in $\Pi$. 

	\item If there exists exactly one $j \in \{1,2,3\}$ so that $m_j$ is non-zero, then $D$ is a  non-negative combination of $H$ and $H-E_j$ which are both in $\Pi$.
	
	\item Now let us consider the case where at least two $m_j$ are non-zero. As the generators of degree 1 of $\Pi$ are invariant under permutations of $E_1,E_2$, and $E_3
    $, we may assume that $m_3 = \min\{m_i \mid i=1,2,3\}$. Then $m_1 > 0$, $m_2 > 0$ and $m_3 < d$ since $2d \geq m_1 + m_2 + m_3$. Hence the class
	\[ 	
	(d'; m_1',m_2',m_3') = 
		(d-1; m_1-1,m_2-1,m_3)
	\]
    satisfies $2d' \geq \sum_{i=1}^{3} m_i'$ and $0 \leq m_i' \leq d'$.

     We remark that $m_1-1$ or $m_2-1$ can be strictly smaller than $m_3$, which explains the use of  the stronger statement here. 
    Consequently, 
    \[
   dH - \sum_{i=1}^3 m_i E_i =  \underbrace{\big( (d-1)H - \sum_{i=1}^2(m_i-1) E_i  - E_3\big)}_{\in \Pi \text{ by induction }} + \underbrace{\big( H -E_1-E_2 \big)}_{\in \Pi}. 
    \]
    is contained in $\Pi$. 
    	\end{itemize}
This proves the claim.

\smallskip
{\em Step 2.} Now let us consider the case $d < m_1 + m_4$. We proceed as in the proof of Theorem \ref{main:eff}. Let $t_{i,j}\coloneqq m_i+m_j-d$ for $1\leq i<j \leq 8$. Then $d<m_1 + m_4$ is equivalent to 
$t_{1,4}>0.$
This implies that $d \geq 3$ and  $m_1\geq m_2+2 \geq 3$.  
Let $a = \max\{i \mid m_i \neq 0\}$. As $d<m_1 + m_4$, we have $a \geq 4$. We consider 
\[
(d';m_1',\ldots, m_{a'}')  \coloneqq
(d;m_1,\ldots,m_a) - (3;3,1^{a-1}),
 \]
 where $a' = \max\{i \mid m'_i \neq 0\}$.
For any $4 \leq a \leq 8$, the divisor class $(3;3,1^{a-1})$ is contained in $\Pi$.
So it is enough to show that $(d';m_1',\ldots, m_8')   \in \Pi$. 

We obtain that $(d';m'_1,\dots,m'_{a'})$ is in standard form and $d'\geq m'_i> 0$ for any $i\in\{1,\dots,a'\}$.
Furthermore, setting $t'_{1,4}\coloneqq m'_1+m'_4-d'$,  we have $t'_{1,4} = t_{1,4}-1$. If $t'_{1,4}>0$, since $(d';m'_1,\dots,m'_{a'})$ is in standard form, we can start our above procedure in {\em Step 2} again and do this until $t'_{1,4}=0$ for some $(d';m'_1,\dots,m'_{a'})$. We are now under the assumption in {\em Step 1} and thus $(d';m'_1,\dots,m'_{a'})\in \Pi$.
\end{proof}

\begin{thm}\label{thm:coverMov}
	Let $\Pi$ be the rational polyhedral cone in $N^1(X)_{\R}$ defined as in Lemma \ref{lem:standardinPi}.
 Then
 \[
 \{ D\in\Mov(X)\mid D \text{ is a rational class}\} \subset {\cal W} \cdot \Pi =\bigcup_{w \in {\cal W}} w \cdot \Pi \subset \MovE(X).
 \]
\end{thm}

\begin{proof}
We first show the second inclusion. Clearly, $\Pi \subset \MovE(X)$ as we can easily verify that every generator of $\Pi$ is a movable effective divisor. Moreover, for a class $D \in \MovE(X)$ we have $w\cdot D \in \MovE(X)$ for any $w \in {\cal W}$ by Theorem \ref{thm:weylelem} and Remark \ref{rem:weylaction}.

For the first inclusion, let us consider a Cartier divisor class $D\in\Mov(X)$. Then by Remark \ref{rem:standform} there exists a divisor class $D_s=(d;m_1,\dots,m_8)$ in standard form and an element $w \in {\cal W}$ such that $D_s =  w \cdot D$. Moreover, by Remark \ref{rem:standardmov} we have $d\geq m_i\geq 0$ for all $i\in\{1,\dots,8\}$. Thus by Lemma \ref{lem:standardinPi}, we obtain
   \[ 
   D =  w^{-1}\cdot D_s \in {\cal W} \cdot \Pi.
   \]
Now let $D\in\Mov(X)$ be a $\Q$-divisor class. Then for sufficiently large and divisible $m\geq 1$, we have that $mD$ is an effective Cartier class in $\Mov(X)$. Thus $D\in {\cal W}\cdot \Pi$.
	\end{proof}

To cover the whole cone $\Mov(X)$ by ${\cal W}$-translates of the rational polyhedral cone $\Pi$, we need to show that ${\cal W}\cdot \Pi$ is convex. To do this, we use the fact that the Tits cone $ T = \bigcup_{w \in {\cal W}} w\cdot\overline{C}$ is convex (see Subsection \ref{subsec:weylgroup}) and we show that ${\cal W}\cdot \Pi$ is precisely the intersection of the Tits cone $T$ with a convex closed set.
 
 \begin{lem}\label{lem:wPiinT}
 	We have  ${\cal W}\cdot \Pi \subset T$. 
 \end{lem}
 \begin{proof}
 	Let us first show that $\Pi \subset T$. Since $\Pi$ is a rational polyhedral cone and $T$ is convex, it is enough to show that any generator of $\Pi$ is in $T$. Recall from the characterization of $\overline{C}$ in \eqref{eq:fundchamber} that all generators of $\Pi$ which are in standard form are contained in $\overline{C}$, and thus clearly in $T$. For each generator $D$ of $\Pi$ which is not in standard form, by Remark \ref{rem:standform} there exists an element of $w \in {\cal W}$ such that $w\cdot D$ is in standard form, hence $D \in w^{-1}\cdot\overline{C} \subset T$. 
 	
 	Since $T$ is invariant under the action of ${\cal W}$, we have $w \cdot \Pi \subset T$ for any $w \in {\cal W}$ and this shows the statement.
 \end{proof}

 \begin{lem}\label{lem:movinT}
 	We have $ \Mov(X) \subset T$.
 \end{lem}
 \begin{proof}
 	Let $D \in \Mov(X) $. Since $\Mov(X)$ is open, we can write
 	\[D = \sum_{i=1}^k \lambda_i D_i,\]
 	where $\lambda_i \in \R_{\geq 0}$ and $D_i\in\Mov(X)$ is a Cartier divisor class for each $i\in\{1,\dots,k\}$. 
  
  Now for each $i\in\{1,\dots,k\}$, since $D_i\in\Mov(X)$ is Cartier, we have $D_i \in \bigcup_{w \in {\cal W}} w\cdot \Pi$ by Theorem \ref{thm:coverMov} and thus $D_i \in T$ by Lemma \ref{lem:wPiinT}. The statement follows from the fact that $T$ is convex. 
 \end{proof}
 
 \begin{rem}
 	We remark that $T$ strictly contains $\Mov(X)$. For example, $E_8\in\overline{C} \subset T$ is effective but not movable. Moreover, since ${\cal W} \cdot E_8$ and $-\frac{1}{2}K_X$ are contained in $T$, the Tits cone $T$ contains the whole effective cone $\Eff(X)$ by Theorem \ref{main:eff}. 
 \end{rem}

 Now we are ready to show that there exists a closed convex set ${\cal B}$ such that 
 \[
 T \cap {\cal B} = \bigcup_{w \in {\cal W}}w\cdot  \Pi.
 \]
  As before we identify $N^1(X)_{\R}$ with $\R^9$: an element $(d,m_1,\ldots,m_8)\in\R^9$ corresponds to a class $dH - \sum_{i=1}^{8}m_i E_i \in N^1(X)_{\R}$. 
 Let 
 \begin{equation}\label{eq:defB}
     B \coloneqq \{(d;m_1,\ldots,m_8) \in N^1(X)_{\R} \mid d -m_i \geq 0 \text{ and } m_i \geq 0 \text{ for all } i=1,\ldots 8\}.
 \end{equation}
 Then $B$ is the intersection of closed half-spaces in $\R^9$, hence it is closed and convex. Furthermore, we have $\Pi \subset B $.

Fix an element $w \in {\cal W}$.
As ${\cal W}$ is generated by the orthogonal reflections in (\ref{eq:weylgroup}), we can consider $w$ as an invertible linear map on $\R^9$. We define
\[
B_w\coloneqq w^{-1} (B).
\]
Then $B_w$ is closed and convex. Moreover, $D\in B_w$ if and only if $w\cdot D\in B$.

 We set
 \[
 {\cal B} = \bigcap_{w \in {\cal W}} B_w.
 \]
 Then  ${\cal B} $ is the intersection of infinitely many closed convex sets, thus it is closed and convex. In particular, $D\in {\cal B}$ if and only if $w\cdot D\in B$ for any $w\in {\cal W}$.
\begin{lem}\label{lem:wPiinB}
We have	${\cal W}\cdot  \Pi \subset {\cal B}$. 
\end{lem}
\begin{proof}
Let $D\in \Pi$ be a generator as in Lemma \ref{lem:standardinPi}. Since the class $D$ contains a movable prime divisor $P$ which is the strict transform of a hypersurface in $\pr^3$, we have $D\in B$. By Theorem \ref{thm:weylelem}, for any $w\in {\cal W},$ the class $w\cdot D$ also contains a movable prime divisor $P'$  which is the strict transform of a hypersurface in $\pr^3$. Hence $w\cdot D\in B$ for any $w\in {\cal W}$. Therefore, $D\in {\cal B}$ and $w\cdot D\in {\cal B}$ for any $w\in{\cal W}$.

Now let $D=\sum_i \lambda_i D_i\in\Pi$, where $\lambda_i\in\R_{\geq 0}$ and $D_i$ is a generator of $\Pi$ for each $i$. For any $w\in {\cal W}$, we have $w\cdot D = \sum_i \lambda_i 
 (w\cdot D_i)$ and $w\cdot D_i\in B$, thus $w\cdot D \in B$ by convexity of $B$. Therefore, we obtain $D\in {\cal B}$ and thus $w\cdot D\in {\cal B}$ for any $w\in{\cal W}$.
\end{proof}

\begin{thm}\label{thm:wPiconvex}
	We have \[
	T \cap {\cal B} = \bigcup_{w \in {\cal W}}w\cdot  \Pi.
	\] In particular, ${\cal W}\cdot  \Pi$ is convex.
\end{thm}
\begin{proof} 
By Lemmas \ref{lem:wPiinT} and \ref{lem:wPiinB}, we have $\bigcup_{w \in {\cal W}}w\cdot  \Pi\subset T \cap {\cal B} $. It remains to show the other inclusion.

\smallskip
{\em Step 1.} In this step we show that $\overline{C} \cap {B} \subset \Pi$. 

 Let $D\in \overline{C} \cap B$. 
 We first assume that $D$ is a Cartier divisor class and we can write it as $(d;m_1,\ldots,m_8)$. By (\ref{eq:fundchamber}) and (\ref{eq:defB}), we have 
	\begin{itemize}[leftmargin=*]
		\item $d \geq m_i \geq 0$ for all $i\in\{1,\dots,8\}$,
		\item $2d  \geq \sum_{i=1}^4 m_i$,
		\item $m_i \geq m_{i+1}$ for all $i\in\{1,\dots,7\}$.
	\end{itemize}
Hence $D$ is in standard form with $d\geq m_i\geq 0$ for all $i\in\{1,\dots,8\}$ and by Lemma \ref{lem:standardinPi}, we obtain $D \in \Pi$. For a $\Q$-divisor $D\in \overline{C} \cap B$, we have that $mD$ is Cartier for sufficiently large and divisible $m$, thus $D\in \Pi$. Now for an $\R$-divisor $D\in\overline{C} \cap B$, there exists a sequence of $\Q$-divisors $(D_n)_{n\in\N}$ in $\overline{C} \cap B$ converging to $D$. Since $D_n\in\Pi$ for any $n\in\N$ and $\Pi$ is closed, this implies $D \in \Pi$. 

\smallskip
{\em Step 2.} Let $D \in T \cap {\cal B}$ be an $\R$-divisor.  Since $D\in T = \bigcup_{w \in {\cal W}} w\cdot\overline{C}$, there exists an element $w \in {\cal W}$ and an $\R$-divisor $D'\in\overline{C}$ such that $w\cdot D' = D$. But $D \in {\cal B}$ implies in particular $w^{-1}\cdot D \in B$. 
Consequently, \[D'=w^{-1}\cdot D\in \overline{C} \cap B \subset \Pi\] by \textit{Step 1}, and thus $D = w\cdot D' \in w\cdot \Pi$.
\end{proof}

As an immediate consequence, we obtain the following result.

\begin{cor}\label{cor:coverMov}
    Let $\Pi$ be the rational polyhedral cone in $N^1(X)_{\R}$ defined as in Lemma \ref{lem:standardinPi}. Then 
    \[
   \Mov(X) \subset {\cal W} \cdot \Pi =\bigcup_{w \in {\cal W}} w \cdot \Pi \subset \MovE(X).
    \]
\end{cor}

\begin{proof}
Let $D \in \Mov(X)$. Then we can write 
\[
D = \sum_{i=1}^k \lambda_i D_i
\]
with $\lambda_i \in \R_{\geq 0}$ and $D_i \in \Mov(X)$ is a Cartier divisor. By Theorem \ref{thm:coverMov} we have that each $D_i$ is contained in ${\cal W} \cdot \Pi$. Now, as ${\cal W} \cdot \Pi$ is convex by Theorem \ref{thm:wPiconvex}, we have $D \in {\cal W} \cdot \Pi$.
\end{proof}

Finally, we have:
\begin{proof}[Proof of Theorem \ref{main:mov}]
    The first statement of part (ii) follows from Lemma \ref{lem:notratpol}.

    To show that the action of ${\cal W}$ on $\MovE(X)$ admits a fundamental domain, we apply Theorem \ref{thm:looi}
to $V = N^1(X)_{\R}$, $C = \Mov(X)$ and $G = {\cal W} \subset \GL(N^1(X)_{\R})$. Then $\overline{C} = \MovC(X)$ and
$C_+=\MovP(X).$

By Corollary \ref{cor:coverMov} there exists a rational polyhedral cone $\Pi \subset C_{+}$ with $C \subset G \cdot \Pi$.  Indeed, 
$\Pi \subset \MovP(X)$ as every generator of $\Pi$ is  movable and rational. 
Thus, the second condition of Theorem \ref{thm:looi} is satisfied and there exists a rational polyhedral cone $\widetilde{\Pi}$ which is a fundamental domain for the action of ${\cal W}$ on $\MovP(X)$:
\[
\MovP(X) = {\cal W} \cdot \widetilde{\Pi}.
\]
Then, combining Corollary \ref{cor:coverMov} and the first condition in Theorem \ref{thm:looi}, we obtain
\[
\MovE(X) \supset {\cal W} \cdot \Pi = \MovP(X) = {\cal W} \cdot \widetilde{\Pi}.
\]
Furthermore, by Proposition \ref{prop:main} we obtain $\MovE(X) = \MovP(X)$. Hence, $\widetilde{\Pi}$ is a fundamental domain for the action of ${\cal W}$ on $\MovE(X)$ and this finishes the  proof of part (ii). Since the cone $\Eff(X)$ is closed by Proposition \ref{prop:effconeclosed}, we have $\MovE(X)=\MovC(X)$ and thus statement (i) follows.
\end{proof}

\section{Other examples}

In this section, we give some further examples of log canonical Calabi-Yau pairs and study their cones of effective curves.

Our first example is a smooth threefold which is isomorphic in codimension one to $\pr^3$ blown up at eight very general points. More precisely, the threefold in the example below is obtained by flopping three of the six rational curves that are flopped by the standard Cremona transformation described in Subsection \ref{subsec:cremona}.

\begin{exa}{\rm (\cite[Subsection 5.1]{LesOtt16})}\label{ex:blpproduct}
    Let $Y \rightarrow \pr^1\times\pr^1\times\pr^1$ be the  blow up of $\pr^1\times\pr^1\times\pr^1$ at $6$ very general points $q_1,\dots,q_6$. Let $H_i$ be the pullback of the hyperplane class of $\pr^1$ under $i$-th projection $p_i\colon \pr^1\times\pr^1\times\pr^1\to \pr^1$ for $i=1,2,3$, and let $E_j\subset Y$ be the exceptional divisor over the point $q_j$ for $j=1,\dots,6$. Then the anticanonical class $-K_Y$ is given by 
    \[
    -K_Y = 2\left(\sum_{i=1}^3 H_i-\sum_{j=1}^6 E_j\right).
    \]
    It follows from \cite[Subsection 5.1, third paragraph]{LesOtt16} that $-K_Y$ is nef and not semiample, and the toric fans diagram therein shows that $Y$ is obtained from $X$ by flopping the three curves which are the strict transforms of the lines through one point and three others among the eight. Taking a general member $D\in|{-}K_Y|$, we obtain a log canonical Calabi-Yau pair $(Y,D)$.
\end{exa}

\begin{prop}
    In Example \ref{ex:blpproduct} the cone $\overline{\NE}(Y)$ is a rational polyhedral cone.
\end{prop}
\begin{proof}
    We follow the same strategy as in the proof of \cite[Theorem 3.1(6)]{CPS14}. Although \cite{CPS14} deals with the blowup of $\pr^1\times\pr^1\times\pr^1$ at $6$ special points such that the obtained threefold has nef and semiample anticanonical divisor, the same argument works here. We indicate below the part of their proof which is needed in our case.

    Let $\ell_i$ be the class dual to $H_i$ for $i=1,2,3$ and let $e_j$ be the class of a line in $E_j$ for $1\leq j \leq 6$. Consider the subcone ${\cal C}$ of the closed cone of curves $\overline{\NE}(Y)$ spanned by the curve classes:
    \[
    \ell_i, e_j, \ell_i- e_j, \text{ for }i=1,2,3 \text{ and } 1\leq j\leq 6.
    \]
    We can easily verify that the following divisor classes span the dual cone ${\cal C}^{\vee}$ of ${\cal C}$:
    \begin{equation}\label{eq:nefdivprod}
        H_1,H_2,H_3, H_1+H_2+H_3 - \sum_{j\in J}E_j, \text{ for }J\subset\{1,\dots,6\} \text{ with }1\leq |J|\leq 6.
    \end{equation}
    Thus we have $\Nef(Y)\subset {\cal C}^{\vee}$.

    It remains to show that every divisor class in (\ref{eq:nefdivprod}) is nef, so that ${\cal C}^{\vee}\subset \Nef(Y)$, and thus we conclude that $\Nef(Y)= {\cal C}^{\vee}$ and $\overline{\NE}(Y)= {\cal C}$ are rational polyhedral cones.

    Note that $H_1,H_2,H_3$ are nef, being pullbacks of ample divisors by morphisms. Now consider a subset $J\subset\{1,\dots,6\}$ with $1\leq |J| \leq 6$ and a divisor
    \begin{equation}\label{eq:anefclass}
         D_J \coloneqq H_1+H_2+H_3 - \sum_{j\in J}E_j = -\frac{1}{2}K_Y + \sum_{k\not\in J} E_k.
    \end{equation}
    Assume by contradiction that $D_J$ is not nef. Then there exists an irreducible curve $C\subset Y$ such that $D_J\cdot C<0$. Since $-K_Y$ is nef, by (\ref{eq:anefclass}) there exists an exceptional divisor $E_{k_0}$ with $k_0\not\in J$ such that $E_{k_0}\cdot C<0$. Hence $C\subset E_{k_0}$ and we obtain that $C$ is proportional to $e_{k_0}$. But $E_{k_0}\cdot e_{k_0}=-1$, $E_k\cdot e_{k_0}=0$ for $k\neq k_0$ and $-\frac{1}{2}K_Y\cdot e_{k_0}=1$, and thus we obtain $D_J\cdot C =0$ by (\ref{eq:anefclass}). This gives a contradiction and finishes the proof.
\end{proof}

Now we give an example of a log canonical Calabi-Yau pair whose cone of effective curves has infinitely many extremal rays. Note that this example is not isomorphic in codimension one to $\pr^3$ blown up at eight very general points, since they have different Picard numbers.
\begin{exa}{\rm (\cite[Example 3.6]{Xie24})}\label{ex:infiniteKtrivialrays}
    Let $S$ be $\pr^2$ blown up at $9$ very general points such that $-K_S$ is nef, not semiample and the unique member in $|{-}K_S|$ is a smooth elliptic curve. Now define $\mathcal{E}\coloneqq \OO_S\oplus \OO_S(-K_S)$ and $p\colon Y\coloneqq \mathbb{P}(\mathcal{E})\to S$. Then $-K_Y\sim 2\xi$ is nef and not semiample, where $\xi$ is a tautological divisor associated to $\OO_{\mathbb{P}(\mathcal{E})}(1)$. Moreover, a general member in $|{-}\frac{1}{2}K_Y|$ is a smooth surface isomorphic to $S$, and two general members in $|{-}\frac{1}{2}K_Y|$ intersect transversally along a smooth elliptic curve. Furthermore, a general member $D\in |{-}K_Y|$ is reducible, and $(Y,D)$ forms a log canonical Calabi-Yau pair.
\end{exa}
\begin{prop}
    In Example \ref{ex:infiniteKtrivialrays} the cone $\overline{\NE}(Y)$ has infinitely many $K_Y$-trivial extremal rays.
\end{prop}
\begin{proof}
We follow closely the proof of \cite[Lemma 2.6]{CJR08}.
    Let $C$ be a $(-1)$-curve on the surface $S$. Twisting by some line bundle $\mathcal{L}$ on $C$, we may write
    \[
    \mathcal{V}\coloneqq \mathcal{E}|_C\otimes \mathcal{L} = \OO_{\pr^1} \oplus \OO_{\pr^1}(a)
    \]
    for some integer $a\geq 0$.

    Let $S_C \coloneqq p^{-1}(C) \simeq \pr(\mathcal{V})$ and let $\tilde{C}\to C$ be the section in $S_C$ over $C$ corresponding to the quotient $\mathcal{V} \to \OO_{\pr^1} \to 0$. Since $C$ is a $(-1)$-curve on $S$, we have $S_C\cdot \tilde{C} = p^* C\cdot \tilde{C} = C\cdot p_* \tilde{C} = -1$ and thus
    \begin{equation}\label{eq:restrictnormalbundle}
        N_{S_C/Y}|_{\tilde{C}}\simeq \OO_{\pr^1}(-1).
    \end{equation}
    
    Let $C_0$ be a tautological divisor associated to the line bundle $\OO_{\pr(\mathcal{V})}(1)$ and let $\ell$ be a fibre of $p|_{S_C}\colon S_C\to C\simeq \pr^1$.
    Then $\tilde{C}\in |C_0 - a\ell|$. Moreover, by the Grothendieck relation we have
    \begin{equation}\label{eq:Grothendieckrel}
    C_0^2 - C_0\cdot (p|_S)^*\big ( c_1(\mathcal{V})\big ) =0, \text{ i.e.\ } C_0\cdot \tilde{C}=0.
    \end{equation}
    Hence, by the adjunction formula and (\ref{eq:restrictnormalbundle})(\ref{eq:Grothendieckrel}), we obtain
    \begin{align*}
         -K_Y\cdot \tilde{C} &= (-K_{S_C} + S_C)|_{S_C} \cdot \tilde{C} \\
         &= \big ( 2C_0 - (p|_{S_C})^*(K_C+\det\mathcal{V}) \big )\cdot \tilde{C} - 1\\
         &= 2C_0\cdot \tilde{C} - (-2 + a) - 1\\
         &= 1-a.
    \end{align*}
    Since $-K_Y$ is nef and divisible by two in $\Pic(Y)$, we have $a=1$.
    Thus
    $-K_Y\cdot\tilde{C}=0$, and we obtain
    \[
    S_C\simeq \pr\big ( \OO_{\pr^1}\oplus \OO_{\pr^1}(1)\big ) \simeq \mathbb{F}_1
    \]
    with $-K_{S_C}\cdot \tilde{C}=1$, i.e.\ $\tilde{C}$ is the minimal section of the ruled surface $\mathbb{F}_1$.

    Therefore, we have $N_{\tilde{C}/{S_C}}\simeq \OO_{\pr^1}(-1)$. Together with (\ref{eq:restrictnormalbundle}), this implies the splitting of the normal bundle sequence
    \[
0 \to N_{\tilde{C}/{S_C}} \to N_{\tilde{C}/Y} \to N_{{S_C}/Y}|_{\tilde{C}} \to 0,
    \]
    and thus $N_{\tilde{C}/Y}\simeq \OO_{\pr^1}(-1)^{\oplus 2}$. Hence, by \cite[(5.5)]{Reid83}, blowing up the curve $\tilde{C}$ in $Y$, the exceptional divisor is isomorphic to $\pr^1\times\pr^1$ and we may blow down in the other direction onto some smooth threefold $Y^+$:
\begin{center}
    \begin{tikzcd}
                             & \Bl_{\tilde{C}}(Y) \arrow[ld] \arrow[rd] &     \\
Y \arrow[rr, "\chi", dashed] &                                          & Y^+,
\end{tikzcd}
\end{center}
where the birational map $\chi$ is a flop (which is a so-called Atiyah flop). Note that the $K_Y$-trivial ray generated by the class $[\tilde{C}]$ is extremal by \cite[Lemma 2.3]{Xie24} applied to the pair $(Y,S_C)$.
Therefore, every $(-1)$-curve on $S$ corresponds to a $K_Y$-trivial extremal ray. Now since there are infinitely many $(-1)$-curves on the surface $S$ (see \cite[Corollary 1]{Dol83}), we obtain infinitely many $K_Y$-trivial extremal rays.
\end{proof}

In view of these examples and the decomposition of the effective movable cone
 \[
\MovE(X) = \bigcup_{\varphi\colon X\dashrightarrow X' \text{ a finite sequence of flops}} \varphi^* \big((\NefE(X') \big).
\]
from Proposition \ref{prop:main}, it is natural to ask the following question:

\smallskip
{\bf Question.} Let $X$ be the blowup of $\pr^3$ at $8$ very general points and let $X\dashrightarrow X'$ be a finite sequence of flops to another smooth projective threefold $X'$. Is the nef cone $\Nef(X')$ a rational polyhedral cone?

\bibliographystyle{alpha}
  \bibliography{biblio.bib}

\end{document}